\mag=\magstep1 
\input epsf 
\documentclass[10pt,a4paper]{amsart} 
\usepackage{amsmath}	
\usepackage{amssymb,latexsym} 
\usepackage{graphicx} 
\usepackage{color}
\usepackage{enumitem}
\usepackage{datetime}
\usepackage{comment}
 
\def\Cal{\mathcal}

\def\ot{\leftarrow} 
\def\<<{\langle } 
\def\>>{\rangle } 
\def\frakl{l}
\def\Im{\operatorname{Im}}

\numberwithin{equation}{section} 
\newtheorem{theorem}{Theorem}[section] 
\newtheorem{proposition}[theorem]{Proposition} 
\newtheorem{corollary}[theorem]{Corollary} 
\newtheorem{definition}[theorem]{Definition} 
 
\newtheorem{remark}[theorem]{Remark} 
\newtheorem{lemma}[theorem]{Lemma}

\def\<{\langle} 
\def\>{\rangle} 
\def\Re{\operatorname{Re}}
\excludecomment{donotforget}
\allowdisplaybreaks[4]

\begin{document}
\title[Period map of triple coverings]
{Period map of triple coverings of $\bold P^2$ and mixed Hodge structures}

\author{K. Matsumoto}
\address[Matsumoto]{
Department of Mathematics,
Hokkaido University, 
Sapporo 060-0810, Japan
}
\email{matsu@math.sci.hokudai.ac.jp}
\author{T. Terasoma}
\address[Terasoma]{
Graduate school of Mathematical Sciences, 
The University of Tokyo, Tokyo 153-8914 Japan
}
\email{terasoma@ms.u-tokyo.ac.jp}

\maketitle

\medskip

\begin{abstract}
We study a period map for triple coverings of $\bold P^2$
branching along special configurations of $6$ lines.
Though the moduli space of special configurations is
a two dimensional variety,
the minimal models of the coverings form a one
parameter family of K3 surfaces.
We extract extra one dimensional
information from the mixed Hodge structure 
on the second relative
homology group.
We define the period map from the moduli space
of marked configurations to the domain $\bold B\times \bold C^2$,
where $\bold B$ is the right half plane, and give a defining equation of
its image by a theta function.
We write down the inverse of the period map using theta functions.
\end{abstract}

\setcounter{tocdepth}{1}

\tableofcontents

\setcounter{section}{0}

\section{Introduction}

\subsection{Introduction}
Period integrals of cyclic coverings 
of $\bold P^2$ branching along configurations
of several lines satisfy
a hypergeometric system of linear differential equations.
In the paper \cite{MSTY}, we study 
several examples of {\it reducible} hypergeometric systems of 
differential equations.
We treat a special case where the branch index is equal to 
$(\dfrac{1}{3},\dfrac{1}{3},\dfrac{1}{3},
\dfrac{1}{3},\dfrac{1}{3},\dfrac{4}{3})$ with
the notation in \cite{MSTY} and the branching lines
$\ell_1, \dots, \ell_6$ satisfy the following
conditions.
\begin{enumerate}
 \item The intersection $\ell_i\cap \ell_j \cap \ell_k$
is a point for $(i,j,k)=(1,3,6), (2,4,6)$.
\item
The intersection $\ell_i\cap \ell_j \cap \ell_k$ is an empty set
if $1\leq i<j<k\leq 6$ and $(i,j,k)\neq (1,3,6),(2,4,6)$.
\end{enumerate}
A configuration of $6$ lines $\vec\ell=\{\ell_1, \dots, \ell_6\}$ 
satisfying the above
condition is simply called a {\it special configuration}.
By changing projective coordinates of $\bold P^2$,
we normalize $\ell_i$ as
\begin{align*}
& \ell_1=p,\quad \ell_2=q,\quad 
\ell_3=1-p,\quad \ell_4=1-q,
 \\
\nonumber
&\ell_5=1-x_1p-x_2q,
\end{align*}
and $\ell_6$ is the infinite line,
where $p$ and $q$ are inhomogeneous coordinates of $\bold P^2$.
We identify the moduli space $\Cal M$ of special configurations of $6$ lines
with 
$$
\{(x_1,x_2)\in \mathbf{C}^2 \mid x_1x_2(x_1-1)(x_2-1)(x_1+x_2-1)\ne 0\}
$$
by considering the 
the condition (2) for the normalized lines.

For an element $\vec\ell\in \Cal M$, 
we have a branched cyclic triple covering $X=X_{\vec\ell}$ of 
$\mathbf{P}^2$ defined by
$$
X:z^3=p^{-2}q^{-2}(1-p)^{-2}(1-q)^{-2}(1-x_1p-x_2q)^{-2}.
$$
The period integrals of $X$ satisfy
Appell's hypergeometric
system $F_2$ of differential equations
with the parameters
$(a;b_1,b_2;c_1,c_2)=(2/3;1/3,1/3;2/3,2/3)$.

In this paper, we study the period 
map for a family $\{X_{\vec\ell}\}_
{\vec\ell \in \mathcal{M}}$.

One can see that the minimal compact smooth model $\widetilde{X}$ 
of $X$ is a K3 surface.
Let $E_2$ be the divisor of $\widetilde{X}$
lying over the intersection point $\ell_2\cap \ell_4\cap \ell_6$.
Let 
$\rho$ be an automorphism of $X$ defined by $(p,q,z)\mapsto (p,q,\omega z)$
where 
$\omega=\dfrac{-1+\bold i\sqrt{3}}{2}$.
The invariant part of 
the relative cohomology 
$H_2(\widetilde{X},E_2,\bold Z)$
under the action of $\rho$
is denoted by 
$H_2(\widetilde{X},E_2,\bold Z)^{\rho}$.
Then the quotient group
$$
H_2(\widetilde{X},E_2,\bold Z)_{\rho}=
H_2(\widetilde{X},E_2,\bold Z)/H_2(\widetilde{X},E_2,\bold Z)^{\rho}
$$
becomes a free $\bold Z[\rho]$-module of rank $3$. Here $\bold Z[\rho]=\bold Z\oplus \bold Z\rho$,
with the relation $\rho-2+\rho+1=0$.

By Deligne \cite{D}, $H_2(\widetilde{X},E_2,\bold Z)_{\rho}$
is equipped with a natural mixed Hodge structure,
whose weight $(-2)$-part is identified with $H_2(\widetilde{X},\bold Z)_{\rho}$.
The module $H_2(\widetilde{X},\bold Z)_{\rho}$ is identified with
the generic transcendental part and becomes a lattice by 
the intersection form. In other words,
the module $H_2(\widetilde{X},E_2,\bold Z)_{\rho}$ is equipped with
the graded polarized mixed Hodge structure.
In this paper, we treat the period map for the mixed Hodge structure 
$H_2(\widetilde{X},E_2,\bold Z)_{\rho}$ and its inverse map.

To formulate the period map, we define a marking 
(Definition \ref{def:marked config})
of 
the Hodge structure on $H_2(\widetilde{X},E_2,\bold Z)_{\rho}$.
We introduce a standard module $W_{(-1)}=\<B_1, B_2, B_3\>_{\bold Z[\rho]}$ 
equipped with a filtration
$W_{(-2)}=\<B_1,B_2\>\subset W_{(-1)}$ and a symmetric bilinear form on $W_{(-2)}$.
The marking of
$H_2(\widetilde{X},E_2,\bold Z)_{\rho}$ 
is defined as a $\bold Z[\rho]$-isomorphism
$$
\mu:W_{(-1)}\xrightarrow{\simeq}
H_2(\widetilde{X},E_2,\bold Z)_{\rho}
$$
satisfying some properties
(see Definition \ref{def:marked config}).
A marked configuration $(\overset{\to}\ell,\mu)=((x_1,x_2),\mu)$ 
is defined as a pair of a point $\overset{\to}\ell=(x_1,x_2)$ in $\Cal M$ and
a marking $\mu$ of
$H_2(\widetilde{X},E_2,\bold Z)_{\rho}$.
The moduli space of marked configurations is denoted by $\Cal M_{mk}$.

We set $
\bold B=\{\eta\in \bold C\mid\ \Re(\eta)>0\}$ and
$\Cal D=\bold B\times \bold C^2$.
In \S \ref{sec:priod map triple cov p3},
using the markings $\mu$ of the mixed Hodge structures,
we define the period map 
$$
per:\Cal M_{mk}\to \Cal D
$$
as follows.
Let $\chi$ be a character of the group $\<\rho\>$ generated by $\rho \in Aut(X)$ defined by 
$\chi(\rho)=\omega$ and its complex conjugate is denoted by $\overline{\chi}$.
Then the $\chi$-part 
$F^2H^2_{dR,c}(\widetilde{X}-E_2)(\chi)$ (resp.
$\overline{\chi}$-part
$F^1H^2_{dR,c}(\widetilde{X}-E_2)(\overline{\chi})$)
of $F^2H^2_{dR,c}(\widetilde{X}-E_2)$ (resp. 
$F^1H^2_{dR,c}(\widetilde{X}-E_2)$)
is a $1$-dimensional vector space.
We choose a basis
$\xi$ (resp. $\overline{\xi}$) of it.
Let $\<*,*\>$ be the natural pairing 
between 
$H_2(\widetilde{X},E_2,\bold Z)_{\rho}$ and
the de Rham cohomology with compact support $H^2_{dR,c}(\widetilde{X}-E_2)$.
We define
a (unnormalized) period matrix $P((x_1,x_2),\mu)$ by
$$
P((x_1,x_2),\mu)=
\begin{pmatrix}
p_{11} & p_{12}\\
p_{21} & p_{22}\\
p_{31} & p_{32}\\
\end{pmatrix}, \quad 
p_{i1}=\<\mu(B_i),\xi\>,
p_{i2}=\<\mu(B_i),\overline{\xi}\>.
$$
For an element $((x_1,x_2),\mu)\in \Cal M_{mk}$,
we define the period
$
per((x_1,x_2),\mu)=(\eta,z)\in \Cal D
$
by ratios
\begin{align*}
\eta=\dfrac{p_{11}}{p_{21}},\quad
z=(z_1,z_2)=(
 \frac{p_{31}}{p_{21}},
 \frac{p_{32}}{p_{22}}) 
\end{align*}
of entries of the period matrix $P((x_1,x_2),\mu)$.
Note that the map $per$ does not depend on the choice of $\xi$
and $\overline{\xi}$.

To study the period $P((x_1,x_2),\mu)$,
we give two elliptic fibrations $\epsilon_1,\epsilon_2$
on the K3 surface $\widetilde{X}$ in \S \ref{sec:conf of six lines loc system}.
The fibration $\epsilon_1$ (resp. $\epsilon_2$) is an isotrivial family with a finite monodromy group.
The monodromy covering $C\to \bold P^1$ is defined by the minimal covering of $\bold P^1$
such that the fibration $\epsilon_1$ is trivialized
by the base change by the covering map $C\to \bold P^1$.
Then $C$ is a triple covering of $\bold P^1$
depending only on 
$$
t=\dfrac{1-x_1-x_2}{(1-x_1)(1-x_2)},
$$
and its genus is $2$.
Using the fibration $\epsilon_1$, we obtain a rational map
$\lambda:E\times C_t \dashrightarrow \widetilde{X}$.
The image $\epsilon_1(E_2)$ of $E_2$ under the map $\epsilon_1$
is a one point and the inverse image of $\epsilon(E_2)$
in $C$ is denoted by $\Sigma_1$.
The rational map $\lambda$ induces an injection of mixed Hodge structures
\begin{equation}
\label{basic isom of mixed HS}
H_1(C,\Sigma_1,\bold Z)\otimes_{\bold Z[\rho]} H_1(E,\bold Z) \to
H_2(\widetilde{X},E_2,\bold Z)_{\rho}.
\end{equation}
Via this map, we study the structure of the generic transcendental 
lattice $T_X$ of the K3 surface $\widetilde{X}$
in \S \ref{section:transcendental lattice structure}.

In \S \ref{sec:period integral of the mon curve}, we study the relative
homology 
$H_1(C,\Sigma_1,\bold Z)$.
Let $H_0(\Sigma_1,\bold Z)^0$ be the kernel of the natural map 
$H_0(\Sigma_1,\bold Z) \to H_0(C,\bold Z)$.
Then $H_0(\Sigma_1,\bold Z)^0$ is a free $\bold Z[\rho]$-module generated by
$(1-\rho)p_1$, where $p_1$ is an element in $\Sigma_1$.
We have an exact sequence 
\begin{equation}
\label{exact mixed Hodge curve case} 
0\to H_1(C,\bold Z)\to H_1(C,\Sigma_1,\bold Z)\to 
H_0(\Sigma_1,\bold Z)^0\to 0
\end{equation}
arising from the weight filtration
of the mixed Hodge structure on $H_1(C,\Sigma_1,\bold Z)$.
Using the injection
(\ref{basic isom of mixed HS})
and the marking $\mu$,
we obtain
\begin{enumerate}
\item
a symplectic basis 
$\{\alpha_1,\alpha_2, \beta_1, \beta_2\}$
of $H_1(C,\bold Z)$ satisfying $\alpha_1=
\rho(\beta_2),\alpha_2=\rho(\beta_1)$, 
and 
\item
a lifting $\beta_3$ of 
$(1-\rho)p_1$ in $H_1(C,\Sigma_1,\bold Z)$.
\end{enumerate}
Let $F^1H^1_{dR,c}(C-\Sigma_1)$ be the first Hodge filtration
in $H^1_{dR,c}(C-\Sigma_1)$.
Via the exact sequence (\ref{exact mixed Hodge curve case}),
we have an isomorphism for
the Hodge filtrations
\begin{equation*}
F^1H^1_{dR,c}(C-\Sigma_1)\simeq 
F^1H^0_{dR}(C)=
H^0(C,{\it\Omega}^1_C).
\end{equation*}
Let 
$\psi_1', \psi_2'$ be the
basis of
$H^0(C,{\it\Omega}^1_C)$ normalized by
$\displaystyle
\<\beta_i,\psi_j'\>=\delta_{ij},
$
where $\<*,*\>$
is the natural pairing  between
$H_1(C,\Sigma_1,\bold Z)$
and $H^1_{dR,c}(C-\Sigma_1)$.
We define the normalized period matrix $\tau \in M(2,\bold C)$ 
of $C$ and incomplete integrals $\zeta\in \bold C^2$ 
by 
$$
\tau =
\begin{pmatrix}
\<\alpha_1,\psi'_1\>, \<\alpha_1,\psi'_2\> \\
\<\alpha_2,\psi'_1\>, \<\alpha_2,\psi'_2\> \\
\end{pmatrix},\quad
\zeta =
\begin{pmatrix}
 \<\dfrac{\beta_3}{1-\rho},\psi'_1\> , \<\dfrac{\beta_3}{1-\rho},\psi'_2\> \\
\end{pmatrix}.
$$
Then $\tau$ belongs to 
the Siegel upper half space 
$\bold H_2=\{\tau\in M_2(\bold C)\mid \tau=\ ^t\tau, \Im(\tau)>0\}$
and the class of $\zeta$ 
is in the image of Abel-Jacobi map 
$C \to J(C)=\bold C^2/(\bold Z^2\tau+\bold Z^2)$.
Using the fact that the injection (\ref{basic isom of mixed HS})
is a homomorphism of mixed Hodge structures 
with an action of $\rho$, we can express $\tau$ and $\zeta$ as
\begin{align}
\label{intro:modular embedding}
\tau&=\frac{1}{2}\begin{pmatrix}
\sqrt{-3}\eta^{-1}  & -1 \\
-1 & \sqrt{-3}\eta 
\\
\end{pmatrix},
\\
\nonumber
\zeta&=
\frac{1}{2}\left(
(\dfrac{z_1}{1-\omega}-\dfrac{z_2}{1-\omega^2})\eta^{-1},
\dfrac{z_1}{1-\omega}+\dfrac{z_2}{1-\omega^2}
\right),
\end{align}
where $per((x_1,x_2),\mu)=(\eta,z_1,z_2)$.

For an element $(x_1,x_2)\in \Cal M$ satisfying the conditions 
$x_1,x_2 \in \bold R, 0<x_1,0<x_2, x_1+x_2<1$,
we define explicit
relative topological 2-cycles $\Gamma_1, \Gamma_2, \Gamma_3$ and $\Gamma_4$
of $H_2(\widetilde{X},E_1\cup E_2,\bold Z)$ in 
\S \ref{subsec:relative chain on K3}.
Using 
these topological 2-cycles, we define a marking $\mu$ in 
Example \ref{explicit markings and periods}.
For this marking $\mu$ and a suitable choice of $\xi, \overline{\xi}$, 
the pairings $\<\mu(B_i),\xi\>$
and $\<\mu(B_i),\overline{\xi}\>$
are expressed in terms of hypergeometric integrals
(\ref{integration by stnadard branch}).

Let $\Gamma$ be 
the unitary group 
with the coefficients in $\bold Z[\rho]$
for the hermitian matrix 
$U=\begin{pmatrix}0 & 1 \\ 1 & 0\end{pmatrix}$,
and let $G=\Gamma\ltimes \bold Z[\rho]^2$
be the semi-direct product of $\Gamma$ and $\bold Z[\rho]^2$
obtained by the standard right action of $\Gamma$ on 
$\bold Z[\rho]^2$.
Then the group $G$ is identified with a subgroup of $GL(3,\bold Z[\rho])$.
Its principal congruence subgroup of level $(1-\rho)$
is denoted by $G(1-\rho)$.
Then the group $G(1-\rho)$
acts on the moduli space $\Cal M_{mk}$ by the action
on the set of markings. 
Using the relation (\ref{intro:modular embedding}),
we define an embedding
$\jmath_{\Cal D}:\bold B\times \bold C^2\to \bold H_2\times \bold C^2$ 
by $\jmath_{\Cal D}(\eta,z)=(\tau,\zeta)$,
which is equivariant under the group homomorphism $\jmath_G:G(1-\rho)\to Sp_2(\bold Z)\ltimes \bold Z^4$.

In \S \ref{sec:theta-inv}, we prove that
the image 
$per(\Cal M_{mk})\subset \Cal D$ 
under the period map is a 
codimension one complex analytic space $V(\vartheta)$
defined by the zero locus of a theta function $\vartheta$
for the period matrix $\tau$.
As a consequence, we have an isomorphism
$$
\Cal M \simeq V(\vartheta)/G(1-\rho) \subset \Cal D/G(1-\rho).
$$
 By using the embedding $\jmath_{\Cal D}$ and theta functions on 
$\mathbf{H}_2\times \mathbf{C}^2$
with characteristics,
we give the inverse
of the above isomorphism $\overline{per}:\Cal M\to V(\vartheta)/G(1-\rho)$.

\subsection{Notations}
\label{subsec:notation}
\begin{enumerate}
\item
Let $\bold Z[\rho]$ be a commutative ring generated by $\rho$
with a relation $\rho^2+\rho+1=0$. 
The conjugate map $\overline\ :\bold Z[\rho]\to \bold Z[\rho]$
is the ring homomorphism defined by
$\overline{\rho}=\rho^{-1}=-1-\rho$.
We set 
$$
\Re:\bold Z[\rho]\ni
x\mapsto 
\dfrac{x+\overline{x}}{2}
\in\dfrac{1}{2}\bold Z.
$$
\item
Let $\chi$ be a character of the cyclic group $\<\rho\>$ of order three.
For a $\bold C$-vector space $V_{\bold C}$ with an action of
$\<\rho\>$, $V_{\bold C}(\chi)$ denotes the $\chi$-part 
$$
V_{\bold C}(\chi)=\{v\in V_{\bold C}\mid \rho(v)=\chi(\rho)v\}.
$$
of $V_{\bold C}$.
For the conjugate character 
$\overline{\chi}$,
$V_{\mathbf{C}}(\overline{\chi})$ denotes the $\overline{\chi}$-part of 
$V_{\mathbf{C}}$.
\item
For a module $M$
with an action of $\rho$,
$M^{\rho}$ and $M_{\rho}$ denote
the $\rho$-invariant part and
the $\rho$-coinvariant part $M_{\rho}=M/M^{\rho}$ of $M$, respectively.
Then $M_{\rho}$ becomes a $\bold Z[\rho]$-module. 
\item
For a topological space $X$ (resp. a pair of topological spaces $X\supset Y$), 
the $i$-th singular cohomology and homology (resp. relative homology)
with integral coefficients are denoted by $H^i(X)$ and $H_i(X)$ (resp. $H_i(X,Y)$).
The $i$-th de Rham cohomology (resp. de Rham cohomology with compact support)
of an algebraic variety
is denoted by $H_{dR}^i(X)$ (resp. $H_{dR,c}^i(X)$). 
It is a $\bold C$-vector space.
\item
In this paper, a lattice means
a finitely generated free $\bold Z$-module with a symmetric bilinear
form over
$\bold Q$.
For a lattice $(L,\<\ , \ \>)$, $L(m)$ denotes a lattice $(L, m\<\ ,\ \>)$
i.e. the module $L$ with the symmetric bilinear form $m\<\  ,\ \>$.

\end{enumerate}
\subsection{Acknowledgment}
The study of period integrals of open K3 surfaces is
motivated by the previous work \cite{MSTY} on
reducible hypergeometric equation. The authors express their gratitude to
T. Sasaki and M. Yoshida for discussions with them.

\section{Triple covering of $\bold P^2$ branching along special configuration of $6$ lines}
\label{sec:conf of six lines loc system}
\subsection{Moduli spaces of special configurations}
\label{sec:moduli special config}

In this subsection, we define a triple covering $X$ of $\bold P^2$
branching along a special configuration of 6 lines.

Two numbered sets of $6$
lines $(\ell_1,\cdots, \ell_6)$ and 
$(\ell_1',\cdots, \ell_6')$ in $\bold P^2$
are projectively equivalent, if and only if there exists an element $g$ in 
$PGL_3(\bold C)$ such that $\ell_i=g(\ell_i')$ for $i=1, \dots, 6$.
\begin{definition}
A numbered set of $6$ lines $\vec\ell=(\ell_1, \dots, \ell_6)$ in $\bold P^2$
is called a {\it special configuration},
if it satisfies the following conditions:
\begin{enumerate}
\item
the intersection $\ell_i\cap \ell_j\cap \ell_k$ is a point,
if $(i,j,k) = (1,3,6), (2,4,6)$,
\item
$\ell_i\cap \ell_j\cap \ell_k= \emptyset$,
if $1\leq i<j<k\leq 6$, $(i,j,k) \neq (1,3,6),(2,4,6)$.
\end{enumerate}
The set of projective equivalence classes of
special configurations is denoted by
$\Cal M$.
\end{definition}

We can choose an inhomogeneous coordinates $(p,q)$ of $\bold P^2$ 
so that $\ell_6$ is the line at infinity.
We use the same notation $\ell_i$ for the equation of the line $\ell_i$.
We choose an inhomogeneous coordinate $(p,q)$ so that the defining equations
are given by
\begin{align*}
& \ell_1=p,\quad \ell_2=q,\quad 
\ell_3=1-p,\quad \ell_4=1-q,
 \\
&\ell_5=1-x_1p-x_2q.
\end{align*}
Under the above normalization, we have the following identification:
$$
\Cal M=\{(x_1,x_2)\in \bold C^2\mid
x_1,x_2 \neq 0,1, \ x_1+x_2\neq 1
\}.
$$
For an element $\vec\ell\in \Cal M$, we define a cyclic triple covering 
$$
\pi:X \to \bold P^2
$$
branching along the $6$ lines $\ell_1, \dots, \ell_6$ by
\begin{align}
\label{eq:triple-surface}
&X=X_{\vec\ell}: z^3=p^{-2}q^{-2}(1-p)^{-2}(1-q)^{-2}
(1-x_1 p-x_2q)^{-2}. 
\end{align}
The ramification divisor of $\pi$ is given by 
$
D=\cup_{i=1}^6\pi^{-1}(\ell_i).
$

\subsection{K3 surfaces obtained by triple coverings and Picard lattices}
Let $\overset{\to}\ell$ be a configuration in $\Cal M$
and set $p_1=\ell_1\cap \ell_3\cap \ell_6$, 
$p_2=\ell_2\cap \ell_4\cap \ell_6$.
The blowing up of $\bold P^2$ with centers $p_1,p_2$
is denoted by $\widehat{\bold P^2}$ (see Figure \ref{fig:Birat-trans}). 
The exceptional divisors
over $p_1$ and $p_2$ are denoted by $e_1$ and $e_2$.
The pencil with the axis $p_1$ (resp. $p_2$) defines a morphism
$\overline{\epsilon_1}:\widehat{\bold P^2}\to \bold P^1$ which is expressed as
$$
\overline{\epsilon_1}:(p,q)\mapsto p
\quad(\text{resp. }\overline{\epsilon_2}:(p,q)\mapsto q)
$$
by the inhomogeneous coordinates $(p,q)$ of $\bold P^2$.
The morphism 
$$
\widehat{\bold P^2}
\xrightarrow{(\overline{\epsilon_1},
\overline{\epsilon_2})} \bold P^1\times \bold P^1
$$
is identified with the contraction of the proper transform
of the line $\ell_6$ in $\widehat{\bold P^2}$.
The images of $e_1,e_2$ and $\ell_1,\dots, \ell_5$
under the contraction
$\widehat{\bold P^2}\to \bold P^1\times \bold P^1$ 
are also denoted by $e_1,e_2$ and $\ell_1,\dots, \ell_5$.

\begin{figure}[htb]
\includegraphics[width=12cm]{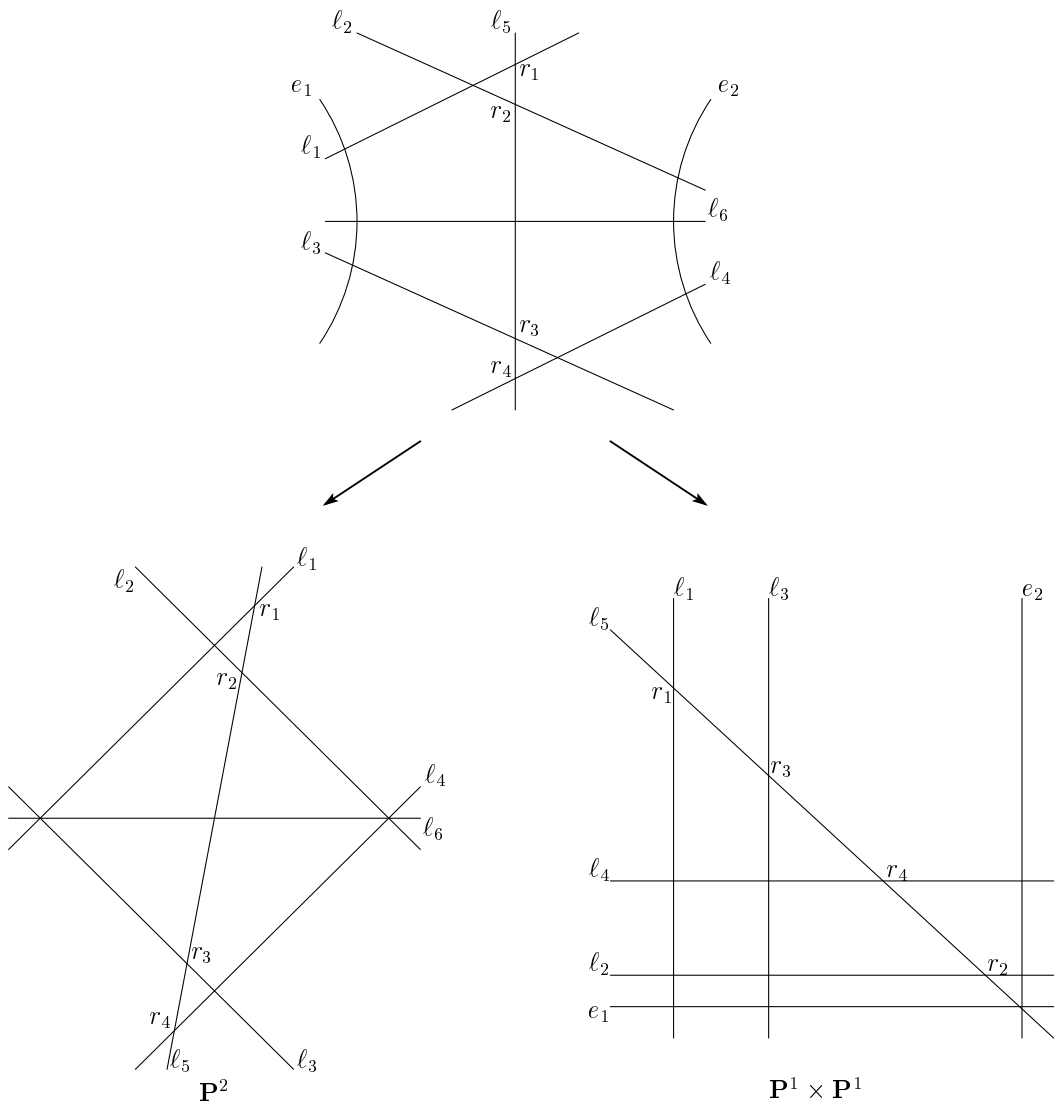}
\caption{Birational transform}
\label{fig:Birat-trans}
\end{figure}

Let $\pi':X' \to \bold P^1\times \bold P^1$ 
be the normalization of $X$
over 
$\bold P^1\times\bold P^1$.
The branch locus of 
$\pi'$ is
the normal crossing divisor
$\ell_1\cup\ell_2\cup\ell_3\cup\ell_4\cup\ell_5$.
For $i=1,2$, the composite 
$$
\epsilon_i=\overline{\epsilon_i}\circ\pi:X'\to \bold P^1
$$
is an elliptic fibration.
We set
$$
{\pi'}^{-1}(e_i)=E_i,\  (i=1,2),\quad
L_i={\pi'}^{-1}(\ell_i) \  (i=1,\dots, 5),
$$
and
$R'=\cup_{i=1}^5 L_i, D'=R'\cup E_1\cup E_2$.
Under the map $\epsilon_1$ (resp. $\epsilon_2$), the images of the curves 
$L_1,L_3,E_2$ (resp. $L_2,L_4,E_1$)
are points, which are denoted by $q_1,q_3$ and $\sigma_2$
(resp. $q_2,q_4$ and $\sigma_1$).
Then we have the following diagram:
$$
\begin{matrix}
&X' & \xrightarrow{\epsilon_2} & \bold P^1 &\ni q_2,q_4,\sigma_1 \\
 &
{\text{\scriptsize$\epsilon_1$}}
 \downarrow \phantom{\epsilon_1}& \\
q_1,q_3,\sigma_2 \in & \bold P^1,
\end{matrix}
$$
and an isomorphism
$
X'-D' \xrightarrow{\simeq} X-D.
$

We show that the minimal resolution 
$\widetilde X$ of $X'$ is a K3 surface.
The $8$ points in $X'$ over the intersection points
in $\ell_1\cup\ell_2\cup\ell_3\cup\ell_4\cup\ell_5$
are $A_2$-singular points on $X'$, and by the ramification formula,
$\widetilde{X}$ is a K3 surface.
For any $\vec\ell \in \Cal M$,
the Picard group contains the classes of $16$ exceptional divisors
arising from $8A_2$ singularities
and the classes of $E_1,E_2$.
Therefore the Picard number of $\widetilde {X}$ is at least $18$.
The dimension of $H^2(X',\bold Q)$ is equal to $22-16=6$.

\begin{definition}[Generic Neron-Severi lattice and generic 
transcendental lattice]
\label{def:generic transcendental space}
Let $\widetilde{X}$ be the minimal resolution of $X'$.
\begin{enumerate}
\item
 We define the generic Neron-Severi lattice $S_X$
of $\widetilde{X}$
 as the primitive closure 
 in
$H^2(\widetilde{X})$
of the submodule generated by 
\begin{enumerate}
 \item 
the classes of
$E_1$, $E_2$;
\item
the classes of $16$ exceptional divisors arising from
the $8A_2$-singular points.
\end{enumerate}
The rank of $S_X$ is $18$.
\item
 We define the generic transcendental lattice $T_X$
of $\widetilde{X}$ as
 the orthogonal complement of $S_X$ in
$H^2(\widetilde{X})$. The rank of $T_X$ is $4$.
\item
Symmetric bilinear forms on $S_X$ and $T_X$
obtained by the intersection 
form on $H^2(\widetilde{X})$
are denoted by $\<\ ,\ \>_S$ and $\<\ ,\ \>_T$, respectively.
Then $S_X$, $T_X$ become primitive lattices in $H^2(\widetilde{X})$.
\end{enumerate}
\end{definition}

\subsection{The elliptic fibration $\epsilon_1$ and the discriminants of lattices}
The elliptic fibration $\widetilde{X} \to \bold P^1$
obtained by $\epsilon_1$ is also denoted by $\epsilon_1$.
We compute the discriminant of $S_X$ using the fibration $\epsilon_1$.
For $i=1, \dots, 5$, the proper transform of $L_i\subset X'$ 
in $\widetilde{X}$ is also denoted by $L_i$.
We set $\widetilde{R}=\cup_{i=1}^5L_i \cup E_1\cup E_2$.
 \begin{proposition}
\begin{enumerate}
 \item 
The discriminant $\delta(S_X)$ of the generic Neron-Severi lattice 
is equal to $9$. As a consequence, the
discriminant $\delta(T_X)$ of the generic transcendental lattice
is equal to $9$.
\item
The image of 
$ 
H_c^2(\widetilde{X}-\widetilde{R})\to 
H^2(\widetilde{X})
$
is identified with $T_X$.
By the Poincar\'e duality, the image of
$H_2(\widetilde{X})\to H_2(\widetilde{X},\widetilde{R})$
is identified with the dual module $T_X^*$ of $T_X$.

\end{enumerate}
 \end{proposition}
\begin{proof}
(1) 
The generic fiber of 
the elliptic fibration 
$\epsilon_1:\widetilde{X}\to \bold P^1$
is an elliptic curve $E$ over the rational function field $K$
 of $\bold P^1$.
The rank of Mordell-Weil group generated by 
$S_X$ is equal to zero. Using the fact that 
the generic fiber of $\epsilon_1$
has a $\bold Z[\omega]$-multiplication, we can show
$$
E(K)_{tor}\simeq (\bold Z/3\bold Z).
$$

 Let $V$ be the sublattice of $S_X$
 generated by components of singular fibers, a generic fiber and
 the zero section. 
Since the type of singular fibers are
 $2A_2+2E_6$, $V^*/V$ is isomorphic to $(\bold Z/3\bold Z)^4$.
By the formula
 $E(K)_{tor}\simeq S_X/V $ due to Shioda (\cite{S}),
we have $S_X^*/S_X\simeq (\bold Z/3\bold Z)^2$.
Therefore $\delta(S_X)=9$ and as a consequence, we have $\delta(T_X)=9$.

(2) The first statement follows from the exact sequence for 
cohomologies
$$
H^2_c(\widetilde{X}-\widetilde{R})\to
H^2(\widetilde{X})\xrightarrow{f}
H^2(\widetilde{R}),
$$
where $f$ is obtained by the intersections with irreducible components
of $\widetilde{R}$.
The second statement follows from the Poincar\'e duality and the first statement.
\end{proof}

\subsection{A trivialization by the monodromy covering}
\label{subsec:a lifting of change var}

Let $t$ be an element in $\bold C-\{0,1\}$,
and $C$ and $E$ be triple coverings of $\bold P^1$
defined by 
\begin{align}
\label{def eq of ellipic curve with w-action}
   &C:w^3=u(1-u)^2(1-tu)^{-1}, \\
\nonumber
& E:{w'}^3=r^{-2}(1-r)^{-2}.
\end{align}
In this subsection, we define a rational map
$$
\lambda:E\times C \dashrightarrow X'.
$$ 

Before defining the rational map $\lambda$, we define two birational
maps $f$ and $f'$ from $\bold C^2$ to itself.
\begin{enumerate}
\item
We define
a birational map $f$ by 
\begin{equation} 
\label{eq:changevar1}
f:(p,r)\mapsto (p,q)=(p,\dfrac{(1-x_1p)r}{rx_2-x_2+1-x_1p}). 
\end{equation}
Then the rational map $f$ preserves the structure of the fibration
arising from the first projection.
The proper transform of $\ell_2,\ell_4$ and $\ell_5$ are 
defined by $r=0,1$
and $\infty$, respectively.

\item
We define the birational map $f'$
by
\begin{equation}
 \label{u as a function of r}
  f':(u,r)\mapsto (p,r)=
(\dfrac{-u}{1-x_1-u},r).
 \end{equation}
\end{enumerate}

Then the correspondence of the variables $p$ and $u$
are given by the following table,

$$
\begin{array}{|c||c|c|c|c|c|}
\hline
u & 0  &1-x_1 & 1 & \dfrac{1}{t}&\infty\\[3mm]
\hline
p & 0  & \infty &\dfrac{1}{x_1}& \dfrac{1-x_2}{x_1}& 1\\[3mm]
\hline
\text{fiber of }\epsilon_1 & L_1  & E_2 & & 
& L_3
\\[1mm]
\hline
\end{array}
$$
where
\begin{equation}
\label{definition of t}
t=\dfrac{1-x_1-x_2}{(1-x_1)(1-x_2)}
\end{equation}
is the cross-ratio of the four points $r_1,r_2,r_3,r_4$ on the
line $\ell_5$ (see Figure \ref{fig:Birat-trans}).
We can easily see that 
via the birational map $f\circ f'$,
the open set 
$$
(\bold P^1-\{0,1,\infty\})\times (\bold P^1-\{0,1,\infty,\dfrac{1}{t}\})
\subset \bold P^1\times \bold P^1=\{(r,u)\}
$$
is mapped
isomorphically onto the open set 
$$
\bold U=
\{(p,q)\in \bold P^1\times \bold P^1\mid
(p,q)\notin (\ell_1\cup \ell_2\cup \ell_3\cup \ell_4\cup \ell_5), p\neq
\dfrac{1}{x_1},\dfrac{1-x_2}{x_1}\}.
$$
We set 
\begin{align*}
&R_C=\{(w,u)\mid u= 0,1,\infty,\dfrac{1}{t}\}\subset C, \quad C^0=C-R_C,\\
 &R_{E}=\{(w',r)\mid r= 0,1,\infty\}\subset
 E,\quad E^0=E-R_{E},\\
&{X'}^0={\pi'}^{-1}(\bold U).
\end{align*}
A covering transformation $\rho$
(resp. $\rho'$)
of $C$ (resp. $E$)
over $\bold P^1$ is defined
by $w\to \omega w$ (resp. $w'\mapsto \omega w'$).
In the next proposition, we choose a branch of
$(1-x_1)^{\frac{1}{3}}(1-x_2)^{\frac{1}{3}}$.
\begin{proposition}
\label{trivialization for monodromy covering}
\begin{enumerate}
\item
We have a morphism
$
\lambda :E^0\times C^0 \to {X'}^0
$ 
defined by 
\begin{equation}
\label{quotient by order 3 action}
 \begin{array}{ccc}
\lambda:E^0\times C^0 &\to &{X'}^0 \\
((w',r),(w,u)) &\mapsto &(z,p,q) \\[2mm]
&&=( 
  \dfrac{(1-x_1)^{\frac{2}{3}}(1-x_2)^{\frac{2}{3}}
  (1-r)(1-tu)ww'}{u(1-u)(1-p)(1-q)(1-x_1p-x_2q)},p,q).
\end{array}
\end{equation}
Here $p=p(r,u),q=q(r,u)$ are rational functions 
on $r$ and $u$ obtained by the relations (\ref{eq:changevar1}) and
(\ref{u as a function of r}).
\item
We define an action $\hat\rho$ on $E\times C$ by
$$
\hat\rho={\rho'}^{-1}\times \rho.
$$
The morphism $\lambda$ factors through the quotient by the action of 
$\hat\rho$
$$
E^0\times C^0 \to (E^0\times C^0)/\<\hat\rho\>,
$$ 
and induces an isomorphism
\begin{equation}
\label{quot isom from product}
(E^0\times C^0)/\<\hat\rho\> \to {X'}^0.
\end{equation}
\end{enumerate}
\end{proposition}
\begin{proof}
The actions of $\rho$ and $\rho'$ are
fixed point free on 
$C^0$ and $E^0$.
Since the product $ww'$ is invariant under the action of
$\hat\rho$,
we have the statement (2).
The above proposition can be checked directly
from the morphism
(\ref{quotient by order 3 action}).
\end{proof}
\begin{definition}[Monodromy covering, monodromy curve]
The covering $C\to \bold P^1$ is the smallest unramified covering of 
$\bold P^1-\{0,1,\infty,\dfrac{1}{t}\}$
trivializing the elliptic fibration $\epsilon_{1}:\widetilde{X} \to \bold P^1$.
This map is called the monodromy covering and the curve $C$ is called the monodromy curve
 of the elliptic fibration
$\epsilon_1:\widetilde{X}\to \bold P^1$.
\end{definition}

\section{Period integrals of the monodromy curve}
\label{sec:period integral of the mon curve}
In this section, we study period integrals of the
monodromy curve $C$ defined in
(\ref{def eq of ellipic curve with w-action}).
Throughout this section, we assume that the parameter
$t$ in (\ref{definition of t})
is a real number and satisfies
$0<t<1$.
\subsection{The homology of the monodromy curve}
\label{relative homology explicit form}
In this subsection, we study the structure of $H_1(C,\bold Z)$
as $\bold Z[\rho]$-module and its intersection form.
We define a symplectic basis $\alpha_1,\alpha_2,\beta_1$ and $\beta_2$
in the homology of $C$ as follows.
We set $\beta_1,\beta_2$ as in the Figure \ref{fig:cycle}.
\begin{figure}[hbt]
\begin{center}
\includegraphics[width=7cm]{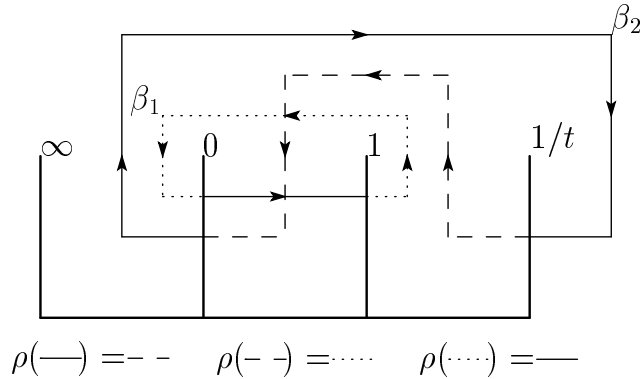}
\end{center}
\caption{Cycles $\beta_1$ and $\beta_2$}
\label{fig:cycle} 
\end{figure}
The first sheet is defined so that $w$ takes
a value in $\bold R_+$ for $0<u<1$.
Paths in the first sheet are written by solid lines.
Then the cycles 
\begin{equation}
\label{symplectic and action of rho}
\alpha_1=\rho(\beta_2),\ \alpha_2=\rho(\beta_1),
\ \beta_1, \ \beta_2
\end{equation}
form a symplectic basis of the first homology 
$H_1(C,\mathbf{Z})$ of $C$; i.e, they satisfy 
\begin{equation*}
\alpha_i\cdot \alpha_j=\beta_i\cdot \beta_j=0,\quad \beta_i\cdot \alpha_j=\delta_{ij}
\end{equation*}
for $1\le i,j \le 2$, where $\delta_{ij}$ is the Kronecker symbol.
We set 
\begin{equation}
\label{equ:definition of matrix U}
I=
\left(
\begin{matrix}
1 & 0 \\
0 & 1
\end{matrix}
\right),\quad
U=
\left(
\begin{matrix}
0 & 1 \\
1 & 0
\end{matrix}
\right).
\end{equation}
The covering transformation $\rho$ induces a linear transformation 
of $H_1(C,\mathbf{Z})$ expressed as  
$$\rho \begin{pmatrix}\alpha\\ \beta \end{pmatrix}
\mapsto 
R\begin{pmatrix}\alpha \\ \beta\end{pmatrix},\quad 
\text{ where }
R=\begin{pmatrix}-I&-U \\U & O \end{pmatrix},
\alpha=\begin{pmatrix}
\alpha_1 \\ \alpha_2
\end{pmatrix},
\beta=\begin{pmatrix}
\beta_1 \\ \beta_2
\end{pmatrix}.
$$

Let $E$ be the elliptic curve defined in (\ref{def eq of ellipic curve with w-action}).
We define a chain $\delta_E$ with the positive direction in $E$ 
by $0\leq r \leq 1, w'\in \bold R_+$ and set
$$
\Delta_E=(1-\rho')\delta_E.
$$ 
Then $\Delta_E$ becomes a cycle in $E$, and 
$H_1(E,\bold Z)$ is a free $\bold Z[\rho']$-module
of rank one generated by the homology class of $\Delta_E$.
Using the Poincar\'e duality, we have the following proposition.

\begin{proposition}
\label{action of rho on curve total dim}
We have $H^1(C,\bold Z)=\bold Z[\rho]\beta_1\oplus \bold Z[\rho]\beta_2$
and $H^1(E,\bold Z)=\bold Z[\rho']\Delta_E$.
The intersection form is written as
\begin{align}
\label{rho invariant symplectic form}
\<m,n\>_C&=\dfrac{2}{3}\Re((\rho-\overline{\rho})\ mU^t\overline{n}), 
\\
\nonumber
\<m',n'\>_E&=\dfrac{2}{3}\Re((\rho'-\overline{\rho'})\ m'\overline{n'}),
\end{align}
where the identification $m=(m_1,m_2)$ (resp. $m'=(m'_1)$) is given by
$m=m_1\beta_1+m_2\beta_2$ (resp. $m'=m_1'\Delta_E$)
with $m_1, m_2\in \bold Z[\rho]$
(resp. $m'_1\in\bold Z[\rho']$).
\end{proposition}

\subsection{Extra involution}
\label{sec:extra involution}
We define an extra involution on $C=C_t$ for $0<t<1$.
\begin{definition}[Extra involution]
\label{def:extra involution}
Let $(1-t)^{1/3}$ be the positive real cubic root of $(1-t)$.
We define an involution $\iota_C:C\to C$, called the extra involution
of $C$,
by setting
\begin{align*}
\iota_C^*(u)&=\frac{1-u}{1-tu}, \quad
\iota_C^*(w)=(1-t)^{1/3}\frac{u(1-u)}{w(1-tu)} .
\end{align*}
\end{definition}
This involution induces a transposition of 
branching points and
boundaries as
$$
0 \leftrightarrow 1,\quad \frac{1}{t}\leftrightarrow \infty,\quad
1-x_1 \leftrightarrow 1-x_2,
$$
on the $u$-coordinate. 
Note that the fiber of $\epsilon_1$ at $u=1-x_1$ is a smooth elliptic curve $E_2$ 
in $\widetilde{X}$.
We have a relation $\rho\circ \iota_C=\iota_C\circ\rho^2$.

Let $\psi_1$ be a holomorphic $1$-form on $C$ defined by
\begin{equation*}
\psi_1=\frac{w}{u(1-u)}du=u^{-2/3}(1-u)^{-1/3}(1-tu)^{-1/3}du,
\end{equation*}
and $\psi_2$ be $\iota_C^*(\psi_1)$.
Then we have
\begin{equation}
\label{action of rho on diff form}
\rho^*(\psi_1)=\chi(\rho)\psi_1,\quad
\rho^*(\psi_2)=\overline{\chi}(\rho)\psi_2,
\end{equation}
and
\begin{equation*}
H^0(C, {\it\Omega}^1_C)(\chi)=\psi_1\bold C,\quad
H^0(C, {\it\Omega}^1_C)(\overline{\chi})=\psi_2\bold C.
\end{equation*}
We set 
\begin{align}
\label{definition of sigma1,2 in C}
\Sigma_i=\{(w,u)\in C\ |\ u=1-x_i\},\quad  (i=1,2).
\end{align}
Since the extra involution preserves the set 
$\Sigma_1\cup \Sigma_2$,
it acts on the relative homology $H_1(C,\Sigma_1\cup
\Sigma_2)$.
We define $\beta_3$ and $\beta_4$ by the following equalities:
\begin{equation}
\label{homological identity 2}
\rho(1-\rho)\frakl_{1-x_1}=\beta_3,\quad
\rho(1-\rho^2)\frakl_{1-x_2}
=\beta_4+\rho\beta_1,
\end{equation}
where $\frakl_{1-x_1}$ and $\frakl_{1-x_2}$   
are paths in the first sheet from $(u,w)=(0,0)$ to the point with $u=1-x_{1}$ 
and that with $u=1-x_2$, respectively.

 \begin{proposition}
\label{action of extra invol on homology}
\begin{enumerate}
 \item 
The images of $\beta_i$ under the 
extra involution $\iota_C$ are given as follows:
$$
\iota_C(\beta_1)=\beta_1,\quad
\iota_C(\beta_2)=-\beta_2,\quad
\iota_C(\beta_3)=\beta_4,\quad
\iota_C(\beta_4)=\beta_3.
$$
\item
We set 
\begin{equation}
\label{definition of y_i}
y_i=\int_{\beta_i}\psi_1 \quad (i=1,\dots,4).
\end{equation}
Then
$$
\int_{\beta_1}\psi_2=y_1,\quad
\int_{\beta_2}\psi_2=-y_2,\quad
\int_{\beta_3}\psi_2=y_4,\quad
\int_{\beta_4}\psi_2=y_3.
$$
\end{enumerate}
 \end{proposition}
 \begin{proof}
(1) The identities in the first statement follows from direct computations.

(2) By the definition of $\psi_2$,
we have
$$
\int_{\beta_i} \psi_2=\int_{\beta_i} \imath_C^*(\psi_1) 
=\int_{\imath_{C*}(\beta_i)} \psi_1.
$$
Thus the statements follows from the results in (1).
\end{proof}

\subsection{Period integral and incomplete integral on the monodromy curve}
We compute the normalized period matrix 
of the curve $C$ in
$$
\bold H_2 =\{\tau \in M_2(\bold C)\mid \ ^t\tau=\tau, \quad \Im(\tau)>0\}
$$
with respect to
the symplectic base $\alpha_1, \dots, \beta_2$ defined in 
\S \ref{relative homology explicit form}
using the result of the last subsection.
We have
$$
\left(
 \begin{matrix}
\int_{\beta_1}\psi_1 &
\int_{\beta_1}\psi_2 \\
\int_{\beta_2}\psi_1 &
\int_{\beta_2}\psi_2 \\
 \end{matrix}
\right)=
\tau_B=\begin{pmatrix}
y_1& y_1 \\
y_2& -y_2 
\end{pmatrix}
$$
by Proposition \ref{action of extra invol on homology}, and 
$$
\left(
 \begin{matrix}
\int_{\alpha_1}\psi_1 &
\int_{\alpha_1}\psi_2 \\
\int_{\alpha_2}\psi_1 &
\int_{\alpha_2}\psi_2 \\
 \end{matrix}
\right)=
\tau_A=\begin{pmatrix}
\omega y_2& -\omega^2 y_2 \\
\omega y_1& \omega^2 y_1 
\end{pmatrix}
$$
by the relations (\ref{symplectic and action of rho})
and  (\ref{action of rho on diff form}).
The normalized period matrix $\tau=\tau_A\tau_B^{-1}$ is equal to
\begin{align}
\label{eq:pt-on-J(C)}
&\tau=\frac{1}{2}\begin{pmatrix}
\sqrt{-3}\eta^{-1}  & -1 \\
-1 & \sqrt{-3}\eta 
\\\end{pmatrix}, 
\quad \eta=\dfrac{y_1}{y_2}\in \bold C.
\end{align}
Note that 
$\tau\in \mathbf{H}_2$ if and only if
$\eta$ is an element in
\begin{equation}
\label{definition of ball B} 
\bold B=\{\eta
\in \bold C\mid \Re(\eta)>0\}.
\end{equation} 
\begin{remark}
\label{rem Schwarz map}
Since $y_1$ and $y_2$ satisfies the hypergeometric 
geometric equation with the parameters $(a,b;c)=(1/3,1/3,1)$,
the (multivalued) function $\eta$ of $t$ yields
the Schwartz map with the Schwartz triangle $(3,\infty,\infty)$ (cf.
\cite{MSTY}).
\end{remark}

We define $\delta_3, \delta_4$ in $H_1(C,\Sigma_1\cup \Sigma_2;\bold Q)$
by the relations
$$
(1-\rho)\delta_3=\beta_3,\quad (1-\rho)\delta_4=\beta_4.
$$
Then the incomplete integrals 
of $(\psi_1,\psi_2)$ along the path
in the second sheet from $P_0$ to the point with $u=1-x_{1}$
are expressed as
\begin{align}
\label{scale change}
\bigg(\int_{\delta_3}\psi_1,\int_{\delta_3}\psi_2\bigg)=
\bigg(\dfrac{1}{1-\omega}y_3, 
\dfrac{1}{1-\omega^2}y_4\bigg). 
\end{align}
Therefore the incomplete integrals for the normalized differential
forms are equal to
\begin{align}
\label{def of zeta}
\zeta=&
\left(\dfrac{1}{1-\omega}y_3, 
\dfrac{1}{1-\omega^2}y_4\right)
\tau_B^{-1} 
\\
\nonumber
=&
\frac{1}{2}\left(
 \frac{y_3}{(1-\omega)y_1}+
 \frac{y_4}{(1-\omega^2)y_1},
 \frac{y_3}{(1-\omega)y_2}-
 \frac{y_4}{(1-\omega^2)y_2}
\right). 
\end{align}
As a consequence, we have the following theorem.
 \begin{theorem}
The normalized period matrix of the monodromy curve defined in
(\ref{def eq of ellipic curve with w-action})
for the symplectic basis defined in
\S \ref{relative homology explicit form} is given in
(\ref{eq:pt-on-J(C)}), where $y_1,y_2$
are defined in (\ref{definition of y_i}).
The incomplete integral for the normalized differential forms
along $\delta_3$ is given by 
$\zeta$ in (\ref{def of zeta}).

 \end{theorem}

\section{Transcendental lattice of K3 surface $\widetilde{X}$}
\label{section:transcendental lattice structure}

In this section, we study the transcendental lattice of
$\widetilde{X}$ for $X=X_{\overset{\to}\ell}$
for $\overset{\to}\ell \in \Cal M$
using the map $\lambda$ defined in 
Proposition \ref{trivialization for monodromy covering}.

\subsection{Transcendental lattices of $\widetilde{X}$ and $E\times C$}
We introduce the conjugate $\bold Z[\rho]$-module structure
on $H^1(E)$ by setting $\rho\cdot m={\overline{\rho'}}m$.
The module $H^1(E)$ with the conjugate $\bold Z[\rho]$-action
is denoted by $\overline{H}^1(E)$.
Using this action, we consider the tensor product
$\overline{H}^1(E)\otimes_{\bold Z[\rho]}H^1(C)$.
We have a relation ${\rho'}^{-1}\otimes 1=1\otimes \rho$
on $\overline{H}^1(E)\otimes_{\bold Z[\rho]}H^1(C)$.
 \begin{definition}
  \begin{enumerate}
   \item
	Let
  $T_{E\times C}$ be the orthogonal complement of
the submodule generated by $E\times \{pt\}$
and $\{pt\}\times C$
in $H^2(E\times C)$.
It is isomorphic to $H^1(E)\otimes H^1(C)$.
The intersection form on $H^1(E)\otimes H^1(C)$ is obtained by the restriction of
that on $H^2(E\times C)$. It is equal to the tensor product of the 
cup products on $H^1(E)$ and $H^1(C)$.
\item
   We define a sub-lattice $KT_{E\times C}$
  of
  $H^1(E)\otimes
 H^1(C)$ by
$$
KT_{E\times C}=
\ker\bigg(H^1(E)\otimes
 H^1(C)\to
\overline{H}^1(E)\otimes_{\bold Z[\rho]}
H^1(C)\bigg).
$$
     \end{enumerate}
It is easy to see that $(1\otimes \rho)v\in KT_{E\times C}$
if $v\in KT_{E\times C}$. By this action,
$KT_{E\times C}$ is a $\bold Z[\rho]$-module.
\end{definition}

\begin{proposition}
\label{prop K3 and ExC transcendental}
\begin{enumerate}
\item
 The discriminant of
      $KT_{E\times C}$ is equal to $9$.
The symmetric bilinear form on 
$KT_{E\times C}$ induced from the intersection form
is isomorphic to $A_2\oplus A_2(-1)$.
Here the Grammian matrix of the $A_2$ lattice is given by
$\begin{pmatrix}2 & -1 \\-1 & 2 \end{pmatrix}$.
 \item
      \label{item nat incl E times C}
  The morphism $\lambda^*$ induces the following natural inclusions:
\begin{align}
\label{product structure, isom on cohomology}
T_{X}
\subset KT_{E\times C}\subset T_{E\times C}.
\end{align}
The modules $T_{X}$ and
$KT_{E\times C}$
are free $\bold Z[\rho]$-modules of rank two.
   \item
We introduce a symmetric bilinear form $\<\ ,\ \>_{E\times C}$
	on $T_X$ by the restriction of the intersection paring
	on $H^2(E\times C)$ via the inclusions
$$
T_X\subset KT_{E\times C} \to H^1(E)\otimes
H^1(C) \to H^2(E\otimes C).
$$
Then we have 
$$
\<\ ,\ \>_{E\times C}=3\<\ ,\ \>_{T}.
$$
 \item
\label{compare tr lat for ExC and X}
      Under the left inclusion of 
(\ref{product structure, isom on cohomology}),
      $T_X$
is equal to $(1\otimes 1-1\otimes \rho)\ KT_{E\times C}$
and $(1+\hat\rho^*+{\hat\rho}^{2*})T_{E\times C}$.
As a consequence, $T_X$ is isomorphic to $A_2\oplus A_2(-1)$
and $(1-\rho)T^*_X=T_X$.
	\end{enumerate}
 \end{proposition}
\def\pic{\operatorname{Pic}}
\begin{proof}
We remark a general fact about the orthogonal complements
of submodules of the Neron-Severi lattices
in the second cohomology groups
of a surfaces.
 Let $g:X_1\to X_2$ be a birational morphism of smooth projective surfaces.
 Let $\Cal S$ be a subset in the Neron-Severi lattice of $X_2$ and
 $\Cal E$ be the set of exceptional divisors for $g$.
 Then the orthogonal complement of $\Cal S$ in $H^2(X_2)$
 is mapped by $g^*$ isomorphically to the orthogonal complement of
 the module generated by $g^*(\Cal S)$ and $\Cal E$.

(1) 
The statement is obtained from
the intersection forms on $H^1(E)$ and $H^1(C)$.
The computation is reduced 
to the case where $t\in \bold R, 0<t<1$, which follows from
Proposition \ref{action of rho on curve total dim}.

 (2)
 Let $\hat\rho$ be the action ${\rho'}^{-1}\times\rho$ on $E\times C$.
Let $\widehat{X}$ be the minimal resolution of
$(E\times C)/\<\hat\rho\>$ and
$(E\times C)\widehat\ $ be the normalization of
$$
(E\times C)\times_{(E\times C)/\<\hat\rho\>}\widehat X.
$$
Then we have the following diagram:
 \begin{equation}
\label{blow up and quotient E times C}
\begin{matrix}
 E\times C & \xleftarrow{g_1} & (E\times C)^{\widehat\ } \\
 \downarrow && 
\phantom{\text{\scriptsize $\lambda^*$}}
\downarrow \text{\scriptsize $\lambda^*$}\\
 E\times C/\<\hat\rho\> & \xleftarrow{f_1} &
\widehat{X}
\\
  & &\phantom{\text{\scriptsize{$g_2$}}}\downarrow \text{\scriptsize{$g_2$}}\\
  & &\widetilde X \\
  & &\phantom{\text{\scriptsize{$f_2$}}}\downarrow \text{\scriptsize{$f_2$}}\\
  & & X'. \\
\end{matrix}
 \end{equation}
 Here the morphisms $f_1$ and $f_2$ are minimal resolutions of
 singularities and
 $g_1$ and $g_2$ are birational morphisms between smooth projective surfaces. 
 By the fact cited above, we have the following diagram:
 \begin{align*}
\begin{matrix}
H^2(\widetilde{X})&\xrightarrow{g_2^*} &
H^2(\widehat{X})
& \xrightarrow{\lambda^*} &
H^2((E\times C)^{\widehat\ })& \xleftarrow{g_1^*} &
H^2(E\times C) 
\\
\cup & &
\cup
&  &
\cup & &
\cup 
\\
T_{X}&\simeq &
 T_{X}
& 
\xrightarrow{\lambda^*} 
&
 T_{(E\times C)^{\widehat\ }}& \simeq &
 T_{E\times C}.
\end{matrix}
\end{align*}
Here,  $T_{(E\times C)^{\widehat\ }}$ is the 
orthogonal complement of the divisors $E\times\{pt\}$,
$\{pt\}\times C$ and exceptional divisors for $(E\times C)\hat\ \to 
E\times C$.

 Since the image of $\lambda^*$ is invariant under the action of
$\hat\rho$ and
$$
0=(\hat\rho)^*v-v=({\rho'}^{-1}\otimes' \rho)^*v-v
=(1\otimes' \rho^2)^*v-v
$$
on
$\overline{H}^1(E)\otimes_{\bold Z[\rho]}H^1(C)$,
the image of $T_X$ under the map $\lambda^*$
is contained in $KT_{E\times C}$.

 (3)
 The generic transcendental lattice $T_X$ is equipped with the
intersection form $\<\ ,\ \>_{T}$ by Definition \ref{def:generic transcendental space}.
On the other hand, the left end
of the isomorphism (\ref{product structure, isom on cohomology})
is equipped with an intersection form $\<\ ,\ \>_{E\times C}$ by restricting that of
 $H^1(E)\otimes H^1(C)$.
 Since the degree of the morphism $(E\times C)^{\widehat\ }\to
 (E\times C)^{\widehat\ }/\<\hat\rho\>$ between smooth projective
 surfaces is three, by the properties of the cup product and the trace map,
we have the statement.

 (4)
 For an element $x\in T_{E\times C}=H^1(E)\otimes H^1(C)$,
 $x+\hat\rho(x)+\hat\rho^2(x)$ is contained in the image of $T_X$.
Since
$H^1(E)\simeq \bold Z[\rho']$
and $H^1(C)\simeq \bold Z[\rho]\oplus \bold Z[\rho]$
 as a $\bold Z[\rho']$-module and a $\bold Z[\rho]$-module,
we have an isomorphism 
$$
(1+\hat\rho+\hat\rho^2)T_{E\times C}\simeq 
(1\otimes 1-1\otimes \rho)KT_{T\times C}.
$$
Therefore we have inclusions
\begin{align*}
\begin{matrix}
(1+\hat\rho+\hat\rho^2)T_{E\times C}&\overset{f}\subset& T_X&\subset &KT_{T\times C}&\subset& T_{E\times C}\\
\Vert & \\
(1\otimes 1-1\otimes \rho)KT_{T\times C}.
\end{matrix}
\end{align*}
Since $\delta((T_X,\<\ ,\ \>)=9$ and the equality 
$\<\ ,\ \>_{E\times C}=3\<\ ,\ \>_{T}$, 
the discriminant $\delta((T_X,\<\ ,\ \>_{E\times T}))$ is equal to 
$3^4\cdot 9$.
By (1), we have $\delta(KT_{E\times C})=9$. Since
$[KT_{E\times C}:(1-\rho)KT_{E\times C}]=9$, the inclusion $f$ 
is an isomorphism.
\end{proof}
\def\pic{\operatorname{Pic}}
As a corollary, we have an explicit basis of
the transcendental lattice and its intersection form.
\begin{corollary}
\label{cor to lattice str of T}
The lattice $T_X$ is freely generated by $(1 +\hat\rho+(\hat\rho)^2)
(\Delta_E\otimes \beta_i)$ $(i=1,2)$ over $\bold Z[\rho]$.
It is isomorphic to $A_2\oplus A_2(-1)$.
\end{corollary}
\begin{proof}
By Proposition \ref{prop K3 and ExC transcendental}
(\ref{compare tr lat for ExC and X}), 
the homomorphism 
$\lambda_*:T_{E\times C} \to T_{X}$ obtained by the Poincar\'e duality
is surjective. 
\end{proof}

\subsection{Relative homologies for divisors in K3 surfaces}

The proper transforms of $E_1, E_2$ in $\widetilde{X}$
are also denoted by $E_1, E_2$.
Then the curves $E_1$ and $E_2$ are isomorphic to $E$
defined in (\ref{def eq of ellipic curve with w-action}).
By the long exact sequence of relative homologies and
the Mayer-Vietoris exact sequence, we have the following 
exact sequences of mixed Hodge
structures 
on integral homologies:
\begin{align}
\label{long exact relative to E1 E2}
0\to 
H_2(E_1)\oplus H_2(E_2)
\to H_2(\widetilde{X}) 
\to 
H_2(\widetilde{X},E_1\cup E_2)\xrightarrow{\partial}
H_1(E_1\cup E_2)\to 0,
 \end{align}
\begin{align*}
0 \to H_1(E_1)\oplus H_1(E_2)\to
H_1(E_1\cup E_2)\to H_0(E_1\cap E_2).
 \end{align*}
We set 
\begin{align*}
&H_2^{(0)}(\widetilde{X},E_1\cup E_2)=\partial^{-1}(H_1(E_1)\oplus H_1(E_2)), \\
&H_2^{(i)}(\widetilde{X},E_1\cup E_2)=\partial^{-1}(H_1(E_i)) \quad (i=1,2),
\end{align*}
where the map $\partial$ is defined in the exact sequence 
(\ref{long exact relative to E1 E2}).
We set
\begin{align*}
M_1&=\dfrac{H_2(\widetilde{X})}
{H_2(E_1)\oplus H_2(E_2)},
\\
M_2&=H_2^{(0)}(\widetilde{X},E_1\cup E_2),
\\
M_3&=H_1(E_1)\oplus H_1(E_2),
\end{align*}
and we have a short exact sequence
$$
0\to M_1 \to M_2\to M_3 \to 0.
$$
\begin{proposition}
\label{relative with two elliptic curve rho coinvariant}
\begin{enumerate}
\item
The module $M_1$
is torsion free.
\item
We have an exact sequence
$$
0\to (M_1)_{\rho} \to (M_2)_{\rho} \to (M_3)_{\rho}\to 0.
$$
 As for the coinvariant part, see
Notations in \S \ref{subsec:notation}.
\item
We have the following isomorphisms.
\begin{align*}
(M_1)_{\rho}=H_2(\widetilde{X})_{\rho}, \quad
(M_2)_{\rho}=H_2^{(0)}(\widetilde{X},E_1\cup E_2)_{\rho}.
\end{align*}
As a consequence, we have the following exact sequences:
\begin{align}
\label{fundamental exact sequence for curves on K3}
&0\to H_2(\widetilde{X})_{\rho} \to
H_2^{(0)}(\widetilde{X},E_1\cup E_2)_{\rho}
\to
 H_1(E_1)\oplus H_1(E_2)\to 0,
 \\
 \nonumber
&0\to H_2(\widetilde{X})_{\rho} \to
H_2(\widetilde{X},E_2)_{\rho}
\to
H_1(E_2)\to 0.
\end{align}
\end{enumerate}
\end{proposition}
\begin{proof}
(1)
By the Poincar\'e duality
$H^2(\widetilde{X})
\simeq H_2(\widetilde{X})$,
the intersection form on
$H^2(\widetilde{X})$ induces 
a symmetric bilinear form on
$H_2(\widetilde{X})$.
The class of $E_i$ is also denoted by $E_i$
for $i=1,2$.
Let $m\in H_2(\widetilde{X})$ and assume that 
$km=aE_1+bE_2$ for some $k\in \bold Z$.
Using the intersection, we have $(km,L_1)=a,(km,L_2)=b$.
Therefore $k(m-(m,L_1)E_1-(m,L_2)E_2)=0$.
Since $H_2(\widetilde{X})$ is torsion free, we have 
$x\in \bold Z E_1\oplus \bold Z E_2$.

(2) Since the fixed part $(M_3)^{\rho}$ of $M_3$ is zero,
the natural map $(M_1)^{\rho}\to (M_2)^{\rho}$ is an isomorphism
and we have the required exact sequence.
The natural homomorphism $M_3\to (M_3)_\rho$ is an isomorphism.

(3) We also have the following exact sequence:
$$
0\to 
H_2(E_2)
\to H_2(\widetilde{X}) 
\to 
H_2(\widetilde{X},E_2)\xrightarrow{\partial}
H_1(E_2)\to 0.
$$
Using this exact sequence and a similar argument in 
Proposition \ref{relative with two elliptic curve rho coinvariant},
we have the following commutative diagram whose rows are exact sequences.
$$
\begin{matrix}
0\to& H_2(\widetilde{X})_{\rho} &\to&
H_2(\widetilde{X},E_2)_{\rho}
&\to&
H_1(E_2)&\to 0 
\\
&\text{\scriptsize$\simeq$}\downarrow 
\phantom{\text{\scriptsize$\simeq$}}
& & \downarrow & & \downarrow
\\
0\to& H_2(\widetilde{X})_{\rho} &\to&
H_2^{(0)}(\widetilde{X},E_1\cup E_2)_{\rho}
&\to&
H_1(E_1)\oplus
H_1(E_2)
&\to 0.
\end{matrix}
$$
\end{proof}

\subsection{Comparison for relative homologies}
By the coordinates $(u,r)$,
the fibers of the fibration $\epsilon_1:X' \to \bold P^1$ at $u=0,\infty,1-x_1$ are equal 
to $L_1,L_3,E_2$. Let $\Sigma_i$ $(i=1,2)$ be the subsets of $C$ defined in (\ref{definition of sigma1,2 in C}).
Via the commutative diagram (\ref{blow up and quotient E times C}), 
we have
\begin{align*}
E\times (C-\Sigma_1)
\ot
(E\times (C-\Sigma_1))^{\widehat\ }
\xrightarrow{\lambda^0}
\widehat{X}-E_2
\to
\widetilde{X}-E_2
\end{align*}
and homomorphisms of Hodge structures:
$$
H_2(E\times C,E\times \Sigma_1)
\ot
H_2((E\times C)^{\widehat\ },E\times \Sigma_1)
\xrightarrow{\lambda^0_*}
H_2(\widetilde{X},E_2).
$$
For a $\<\rho',\rho\>$-module $M$,
we set $M_{\Delta}=M/\<\hat\rho(x)-x\>$, where 
$\hat\rho={\rho'}^{-1}\times\rho $.
If $\hat\rho$ acts trivially on $M$, then $M_{\Delta}=M$.
Then the group $\<\rho',\rho\>/\<\hat\rho\>=\<\rho\>$ acts on the quotient module 
on $M_{\Delta}$.
Let $M, N$ be $\<\rho',\rho\>$-modules
and $f:M\to N$ be a $\<\rho',\rho\>$ homomorphism. Then we have
a homomorphism $f_{\Delta}:M_{\Delta}\to N_{\Delta}$.
The $\rho$ coinvariant $(M_{\Delta})_{\rho}$ of $M_{\Delta}$ is denoted by $M_{(\Delta,\rho)}$.
By applying the operation $(*)_{(\Delta,\rho)}$ to the above homomorphisms, 
we have the following homomorphisms
$$
H_2(E\times C,E\times \Sigma_1)_{(\Delta,\rho)}
\xleftarrow{\simeq}
H_2((E\times C)^{\widehat\ },E\times \Sigma_1)_{(\Delta,\rho)}
\xrightarrow{\lambda^0_*}
H_2(\widetilde{X},E_2)_{\rho},
$$
and the following commutative diagram:

$$
{\small
\begin{matrix}
0\to& \! \! \! \!H_2(E\! \times\!  C)_{(\Delta,\rho)}\! \! \! \! &\to&
\! \! \!\! H_2(E\! \times\!  C,E\! \times\!  \Sigma_1)_{(\Delta,\rho)}
\! \! \!\! &\to& \! \! \! \!
(H_1(E)\otimes H_0(\Sigma_1)^0)_{(\Delta,\rho)}
\! \! \! \!
&\to 0 
\\
 &
\lambda_1\downarrow \phantom{\lambda_1}
& & 
\lambda_2\downarrow \phantom{\lambda_2}
& & 
\lambda_3\downarrow \phantom{\lambda_3}
\\
0\to& H_2(\widetilde{X})_{\rho} &\to&
H_2(\widetilde{X},E_2)_{\rho}
&\to&
H_1(E_2)&\phantom{ .}\to 0 .
\end{matrix}
}
$$
Here $H_0(\Sigma_1)^0$ is the kernel of the map
$$
H_0(\Sigma_1)\to H_0(C).
$$
 \begin{proposition}
\label{prop:invariant in tensor}
\begin{enumerate}
\item
We have the following canonical isomorphisms:
\begin{align*}
&H_2(E\times C)_{(\Delta,\rho)}
\simeq H_1(E)\otimes_{\bold Z[\rho]}H_1(C), \\
&H_2(E\times C,E\times \Sigma_1)_{(\Delta,\rho)}
\simeq H_1(E)\otimes_{\bold Z[\rho]}
H_1(C,\Sigma_1), \\
&(H_1(E)\otimes H_0(\Sigma_1)^0)_{(\Delta,\rho)}
\simeq H_1(E)\otimes_{\bold Z[\rho]} H_0(\Sigma_1)^0.
\end{align*}
\item
The homomorphisms $\lambda_1,\lambda_2$ and $\lambda_3$ are injective 
and their images are equal to
$(1-\rho)H_2(\widetilde{X})_{\rho}$,
$(1-\rho)H_2(\widetilde{X},E_2)_{\rho}$ and
$(1-\rho)H_1(E_2)$, respectively.
\end{enumerate}
 \end{proposition}

\begin{proof}
(1)
The isomorphisms are obtained by a direct computation using the 
structures of $H^1(C)$ and $H^1(E)$ given
in Proposition \ref{action of rho on curve total dim}.

(2)
Via the Poincar\'e duality,
the map $\lambda_*$ is identified with
$1+\hat\rho^*+(\hat\rho^2)^*$ on $H^2(\widetilde{X})$, and
the image of the homomorphism $\lambda_1$ is equal to $T_X$
by Corollary \ref{cor to lattice str of T}.
Since each component of exceptional divisors from $A_2$
singularities is fixed under the action of $\rho$, we
have
$H_2(\widetilde{X})^\rho=S_X$ and
$
H_2(\widetilde{X})_\rho=T_X^*.
$
By Proposition \ref{prop K3 and ExC transcendental}
(\ref{compare tr lat for ExC and X}), 
we have $(1-\rho)T_X^*=T_X$ and
the image of $\lambda_*$ is
$(1-\rho)H_2(\widetilde{X})_{\rho}$
under the natural map $T_X\to H_2(\widetilde{X})_{\rho}$.
The module 
$H_1(E)\otimes H_0(\Sigma_1)^0$ is generated by $\Delta_E\otimes(1-\rho)$
as a $\<\rho', \rho\>$-module.
Its image in $(H_1(E)\otimes H_0(\Sigma_1)^0)_{\Delta}$ is equal to 
$(\rho'-1)\Delta_E\otimes 1$.
Thus the image is isomorphic to $(1-\rho)H_1(E_2)$. The statement for
$\lambda_2$ follows from those of $\lambda_1$ and $\lambda_3$.
\end{proof}
We use the following isomorphisms to compute the intersection
form on $T_X$.
$$
\begin{matrix}
 (1+\hat{\rho} +\hat{\rho}^2)H^1(E)\otimes H^1(C) 
&\xleftarrow[\simeq]{\lambda^*} &T_X
\\
\\
(H_1(E)\otimes H_1(C))_{\rho} &\xrightarrow[\simeq]{\lambda_*} & 
T_X& \subset T^*_X\simeq H_2(\widetilde{X})_{\rho}.
\end{matrix}
$$
The following proposition is used to define markings for
K3 surfaces.
\begin{proposition}
\label{hermitian form using delta E beta}
Let $m,n$ be elements in $H_1(C)$.
We use the identification $H_1(C)\simeq \bold Z[\rho]^2$ given in Proposition 
\ref{action of rho on curve total dim}.
\begin{enumerate}
\item
Let $m', n'$ be elements $T_X$ defined by
$m'=\lambda_*(\Delta_E\otimes m),n'=\lambda_*(\Delta_E\otimes n)$.
Then we have
$$
\<m',n'\>_{X}=2\Re(mU^t\overline{n}).
$$
\item
Let $m'', n''$ be elements $T_X^*=H_2(\widetilde{X})_\rho$ defined by
$m''=\lambda_*(\delta_E\otimes m),n''=\lambda_*(\delta_E\otimes n)$.
Then we have
$$
\<m'',n''\>_{X}=\dfrac{2}{3}\Re(mU^t\overline{n}).
$$
\end{enumerate}
Here the matrix $U$ is given in (\ref{equ:definition of matrix U}).
 \end{proposition}
\begin{proof}
(1) We use the following identities:
\begin{align*}
\<m',n'\>_X
&=\< \lambda_*(\Delta_E\otimes m),\lambda_*(\Delta_E\otimes n)\>_X
\\
&=\< \Delta_E\otimes m,\lambda^*\lambda_*(\Delta_E\otimes n)\>_X
\\
&=\< \Delta_E\otimes m,(1+\hat{\rho}+\hat{\rho}^2)(\Delta_E\otimes n)\>_X
\end{align*}
and
$$
(1+\hat{\rho}+\hat{\rho}^2)(\Delta_E\otimes n)=
\Delta_E \otimes n+
(\rho' \Delta_E) \otimes (\rho^2 n)+
({\rho'}^2 \Delta_E)\otimes (\rho n).
$$
By the definition of intersection form and the orientation of a 
fiber space, we have
$$
\<\Delta \otimes m, \Delta'\otimes m'\>_X=
-\<\Delta,\Delta'\>_E\cdot\<m,m'\>_C.
$$
The proposition follows from Proposition \ref{action of rho on curve total dim}
and direct computation.

Statement (2) follows from the equality $\delta_E=\dfrac{1}{1-\rho'}\Delta_E$
and statement (1). 
\end{proof}

By the duals of de Rham cohomologies, we have the following
isomorphism of de Rham cohomologies:
\begin{equation}
\label{isom for de Rham cohomology}
H^1_{dR}(E)\otimes_{\bold C[\rho]}
H^1_{dR,c}(C-\Sigma_1)\xrightarrow{\lambda^*}
H^2_{dR,c}(X-E_2)_{\rho}.
\end{equation}

\section{Period integral for K3 surfaces in $M \subset\Cal M$.}
\label{sec:perindo integrals if x1,x2 in M}
In this section, we study period integrals of 
a triple covering $X_{\overset{\to}\ell}$,
where $\overset{\to}\ell=(x_1,x_2)$ is an element in 
\begin{equation}
\label{real domain moduli of K3}
M=\{(x_1,x_2)\subset \Cal M\mid\ x_1, x_2\in \bold R,
0<x_1,x_2,x_1+ x_2<1\}.
\end{equation}
\subsection{Relative chains on K3 surfaces 
and period integrals}
\label{subsec:relative chain on K3}
We define relative 2-chains $\Gamma_1, \dots, 
\Gamma_4$ on $X$ by
\begin{align*}
\Gamma_1&=\{(p,q,z)\mid
 0\leq p\leq 1, 0\leq q\leq 1, z \in 
\bold e(0)\bold R_+\}, \\
\Gamma_2&=\{(p,q,z)\mid 
1\leq p, 1\leq q, x_1p+x_2q\leq 1,
 z \in \bold
 e(\frac{2}{3})\bold R_+\}, \\
\Gamma_3&=\{(p,q,z)\mid
p\leq 0, 0\leq q\leq 1, z \in
\bold e(\frac{1}{3} )\bold R_+\}, \\
\Gamma_4&=\{(p,q,z)\mid
 0\leq p\leq 1, q\leq 0,
 z \in 
\bold e(\frac{1}{3} )\bold R_+\};
\end{align*}
see Figure \ref{fig:2-cyccles}.
\begin{figure}[htb]
\includegraphics[width=7cm]{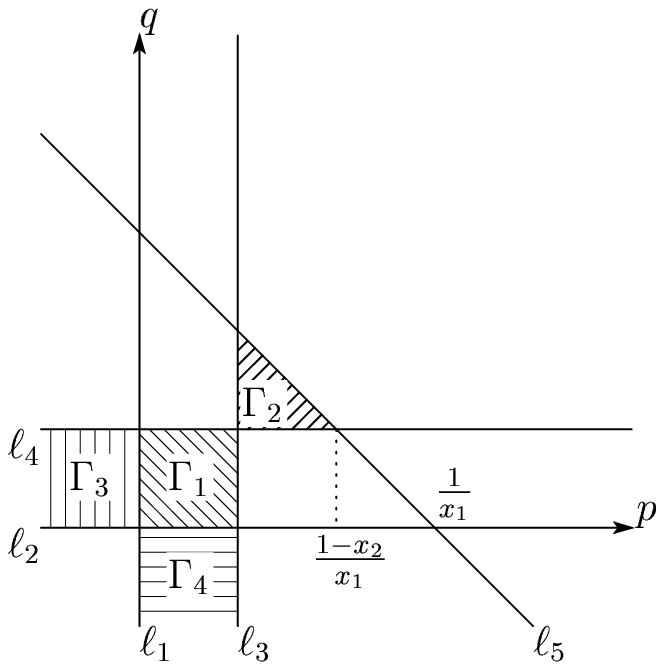}
\caption{Images of $\Gamma_i$ in $\bold P^2$}
\label{fig:2-cyccles}
\end{figure}
Then $(1-\rho)\Gamma_1, \dots, (1-\rho)\Gamma_4$
define elements in the relative homology 
$H_2(X,E_1\cup E_2)_{\rho}.$
A holomorphic two form $\Omega_X$ on $X-D$ given by
\begin{align*}
\Omega_X=&zdp\wedge dq\ \big(=p^{-\frac{2}{3}}q^{-\frac{2}{3}}
(1-p)^{-\frac{2}{3}}(1-q)^{-\frac{2}{3}}
(1-x_1 p-x_2q)^{-\frac{2}{3}}dp\wedge dq\big)
\end{align*}
becomes a global two form on $\widetilde{X}$, 
which is also denoted by $\Omega_X$.
Since the restriction of $\Omega_X$ to $E_1\cup E_2$
is zero,
it defines an element of 
the de Rham cohomology
$
H^2_{dR, c}(\widetilde{X}-E_1\cup E_2)
$ 
with compact support.
The natural pairing 
$$
\<\ ,\ \>:H_2(X,E_1\cup E_2)_{\rho} \otimes
H^2_{dR,c}(X-E_1\cup E_2)_{\rho} \to \bold C,
$$
is defined by period integrals. We define 
functions $\varphi_i$ on $M$ by
\begin{equation}
\label{integration by stnadard branch}
\varphi_i=\<\Gamma_i,\Omega_X \>=\int_{\Gamma_i}\Omega_X,
\quad (i=1,2,3,4).
\end{equation}
Then we have a map 
$$
M \ni \vec\ell\mapsto (\varphi_1,\dots, \varphi_4)
\in\bold C^4.
$$
This map is continued analytically to a multivalued holomorphic
map from $\Cal M$
to $\bold C^4$.

\begin{remark}
The integral $\varphi_i$ ($i=1,2,3,4$) 
satisfies Appell's hypergeometric system $F_2$ of differential equations 
with parameters
$(a,b_1,b_2,c_1,c_2)=(2/3,1/3,1/3,2/3,2/3)$. 
\end{remark}

\subsection{Comparison of Period integrals}
\label{subsec:comp of integrals}
In this subsection, we compute period integrals using the isomorphisms
in Proposition \ref{prop:invariant in tensor}.
Let $(x_1,x_2)$ be an element in $M$ defined in (\ref{real domain moduli of K3})
and set
$$
t=\dfrac{1-x_1-x_2}{(1-x_1)(1-x_2)}.
$$
Then we have
$0<t<1$ since
$0< x_1, x_2,x_1+x_2 <1$.
Let 
$\lambda:E^0\times C^0 \to {X'}^0 $ 
be a morphism defined in Proposition \ref{trivialization for monodromy covering}.
 \begin{proposition}
\begin{enumerate}
\item
The pull back of the differential form
  $\Omega_X$ 
under the covering map $\lambda$ in (\ref{quot isom from product})
is equal to
 \begin{align*}
  \lambda^*\Omega_X
=&(1-x_1)^{-\frac{1}{3}}
  (1-x_2)^{-\frac{1}{3}}
w'dr\wedge \psi_1
\\
=&(1-x_1)^{-\frac{1}{3}}
  (1-x_2)^{-\frac{1}{3}}
r^{-\frac{2}{3}}(1-r)^{-\frac{2}{3}}dr \\
&\wedge
u^{-\frac{2}{3}}(1-u)^{-\frac{1}{3}}
(1-tu)^{-\frac{1}{3}}du.
 \end{align*}
\item
  Via the isomorphism $\lambda$ in (\ref{quotient by order 3 action}), 
The chains $\Gamma_1,\Gamma_2, \Gamma_3$
correspond to
\begin{align*}
\Gamma_1' =& \{(ww',r,u)\in (E_{\omega}\times C)/\<\hat\rho\>
 \mid 0<r<1, -\infty<u<0, ww'\in \bold e(\dfrac{1}{2})\}, \\
\Gamma_2' =& \{(ww',r,u)\in (E_{\omega}\times C)/\<\hat\rho\>
 \mid
 1<r<\infty,\dfrac{1}{t}<u<\infty, ww'\in \bold e(\dfrac{1}{6}) \}, \\
\Gamma_3' =& \{(ww',r,u)\in (E_{\omega}\times C)/\<\hat\rho\>
 \mid
 0<r<1, 0<u<1-x_1, ww'\in \bold e(\dfrac{1}{3})\}.
\end{align*}
\end{enumerate}
 \end{proposition}
We define $1$-chains $\gamma_i$ on $C$ 
by
 \begin{align*}
\gamma_1&=\{(w,u)\in C\mid u\in
  (-\infty,0), w\in \bold e(\dfrac{1}{2})\}, \\
\gamma_2&=\{(w,u)\in C\mid u\in
  (\dfrac{1}{t},\infty), w\in \bold e(\dfrac{1}{6})\}, \\
\gamma_3&=\{(w,u)\in C\mid u\in
  (0,1-x_1), w\in \bold e(\dfrac{1}{3})\}.
 \end{align*}
Then by the covering $\lambda:E\times C \to (E\otimes C)/\<\hat\rho\>$,
we have
\begin{equation}
\label{relation between chain and products}
\Gamma_i'=\lambda(\delta_E\otimes \gamma_i) \quad (i=1,2,3),
\end{equation}
where $\delta_E$ is defined in \S \ref{relative homology explicit form}.
Using the above relation, we have the following theorem.
 \begin{theorem}
\label{thm:second specialization identity 1}
The period integrals $\varphi_i$ $(i=1,2,3)$ are expressed as
$$
  \varphi_i(x_1,x_2)=(1-x_1)^{-\frac{1}{3}}
  (1-x_2)^{-\frac{1}{3}}
B(\dfrac{1}{3},\dfrac{1}{3})\ \varphi'_i(t)
$$
  for $i=1,2,3$,
where
  $\varphi'_i=\varphi'_i(t)$
is
$$
\varphi_i'(t)=\int_{\gamma_i}
u^{-\frac{2}{3}}(1-u)^{-\frac{1}{3}}
(1-tu)^{-\frac{1}{3}}du.
$$
 \end{theorem}

The computation of the integral $\varphi_4(x_1,x_2)$
is reduced to that of $\varphi_3$ by 
exchanging the parameters $x_1 \leftrightarrow x_2$, 
and the variables $p\leftrightarrow q$.
We have
\begin{align*}
 \varphi_4= &B(\dfrac{1}{3},\dfrac{1}{3})
 (1-x_1)^{-\frac{1}{3}}(1-x_2)^{-\frac{1}{3}}
 \varphi_4',
 \end{align*}
where
\begin{align*}
\varphi_4'&=
\int_{\gamma_4}
u^{-\frac{2}{3}}
(1-u)^{-\frac{1}{3}}
(1-tu)^{-\frac{1}{3}}
du, \\
\gamma_4&=\{(w,u) \in C \mid u\in (0,{1-x_2}),
w\in \bold e(\dfrac{1}{3})\}.
 \end{align*}
\subsection{Relations between cycles $\{\gamma_i\}$ and $\{\beta_i\}$ in $C$}
\label{sec: rel between cycles}
We give relations between the bases
$\{\gamma_i\}$ and $\{\beta_i\}$ of
$H_1(C,\Sigma_1\cup
\Sigma_2)_{\rho}$.
The action on $H_1(C,\Sigma_1\cup
\Sigma_2)_{\rho}$ induced by $\rho$ is also denoted as
$\rho$.
We define points $P_0, P_1, P_t$ and $P_{\infty}$
in $(u,w)\in C$ by
\begin{equation}
\label{eq:ramification-pts}
P_0=(0,0),\quad P_1=(1,0),\quad P_t=(1/t,\infty),\quad 
P_\infty=(\infty,\infty).
\end{equation}
Paths connecting $P_0$ with
$P_1, P_t$ and $P_\infty$ in the first sheet are denoted 
by $\frakl_1$, $\frakl_t$ and $\frakl_\infty$, respectively.
Then we have
\begin{equation}
\label{homological identity 1}
(1-\rho^2)\frakl_1=\beta_1,\quad (1-\rho)\frakl_t=\beta_2,\quad 
(1-\rho)\frakl_\infty=-\rho^2\beta_1+\beta_2.
\end{equation}
Recall that $\frakl_{1-x_1}$ and $\frakl_{1-x_2}$   
are paths in the first sheet from $P_0$ to the point with $u=1-x_{1}$ 
and that with $u=1-x_2$, respectively.
The paths $l_{1-x_1}$ and $l_{1-x_1}$ and cycles 
$\beta_1,\beta_2,\beta_4$ satisfy the relations in (3.10).
\begin{proposition}
\label{def:marked-configuration}
In 
$
H_1(C,\Sigma_1\cup
\Sigma_2)_{\rho}$, we have the following identities:
\begin{align}
\label{equation gamma1-4}
 \beta_1&=\rho(1-\rho^2)\gamma_2,\quad
 \beta_2=(1-\rho^2)\gamma_1
+(1-\rho^2)\gamma_2,\\  
\nonumber
\beta_3&=(1-\rho)\gamma_3,\quad
\beta_4=(1-\rho^2)\gamma_4
-\rho^2(1-\rho^2)\gamma_2,
\end{align}
where $\gamma_i$ is defined in \S \ref{subsec:comp of integrals}.
\end{proposition}
 \begin{proof}
The above identities follow from
(\ref{homological identity 2})
and the identities
\begin{align*}
& (1-\rho)\gamma_1=-(1-\rho)\rho\frakl_{\infty},\quad
 (1-\rho)\gamma_2=(1-\rho)\rho(\frakl_{\infty}-\frakl_t),
\\
& (1-\rho)\gamma_3=(1-\rho)\rho\frakl_{1-x_1},\quad
 (1-\rho)\gamma_4=(1-\rho)\rho\frakl_{1-x_2}
\end{align*}
in $H_1(C,\Sigma_1\cup \Sigma_2)$.
 \end{proof}

\begin{proposition}
\label{corresondence of cycles in ExC and X generators}
We set
$$
B_i^*=\lambda(\delta_E\otimes \beta_i)\quad (i=1,2,3).
$$
Then $H_2(\widetilde{X},E_2)_{\rho}$ is freely generated by
$B_1^*, B_2^*$ and $B_3^*$.
We have the following relations
between $B_1^*, B_2^*, B_3^*$ and $\Gamma_1, \Gamma_2, \Gamma_3$:
\begin{align*}
B_1^*&=\rho(1-\rho^2)\Gamma_2,\ 
B_2^*=(1-\rho^2)\Gamma_1
+(1-\rho^2)\Gamma_2,\ 
B_3^*=(1-\rho)\Gamma_3.
\end{align*}
As a consequence,
$H_2(\widetilde{X},E_2)_{\rho}$ is freely generated by
$(1-\rho)\Gamma_1,(1-\rho)\Gamma_2$ and $(1-\rho)\Gamma_3$.
\end{proposition}
\begin{proof}
The first statement of the proposition is a consequence of
Proposition \ref{action of rho on curve total dim},
\ref{prop:invariant in tensor}.
By the equality (\ref{relation between chain and products}),
we have
$$
\Gamma_i=\lambda(\delta_E\otimes \gamma_i)\quad (i=1,2,3),
$$
and the relations 
follow from
(\ref{equation gamma1-4}).
\end{proof}

\section{Period map for marked triple coverings of $\bold P^2$}
\label{sec:priod map triple cov p3}
In this section, we define a marking on a special configuration
of $6$ lines in $\bold P^2$,
and its moduli space $\Cal M_{mk}$.
We define the period map $\Cal M_{mk}\to \bold B\times \bold C^2$ 
from the moduli space of marked configuration to 
a period domain (for the definition of $\bold B$, see 
(\ref{definition of ball B})).

Let $\vec\ell=(x_1,x_2)$ be an element 
in $\Cal M$.
Recall that
the triple covering $X=X_{\vec\ell}$ of $\bold P^2$ and
the triple covering $C=C_{\vec\ell}$ of $\bold P^1$
are defined by the equations (\ref{eq:triple-surface}) and
 (\ref{def eq of ellipic curve with w-action}).
Let $\Sigma_i$ be the subsets of $C$ defined in 
(\ref{definition of sigma1,2 in C}).

\subsection{Level structures of monodromy curves and K3 surfaces}
\label{subsec:level strucutre on K3}
Let $(x_1,x_2)$ be an element in $\Cal M$ (not necessarily contained in the domain $M$).
We define
mod $(1-\rho)$-markings of $H_1(C)$ and $H_2(\widetilde{X})_{\rho}$.

We begin with the mod $(1-\rho)$-marking of $H_1(C)$.
We choose a path $\frakl_1$  (resp. $\frakl_t, \frakl_{1-x_1}$) in $C$
starting from the branching point $u=0$ and
ending with $u=1$ (resp. $u=\dfrac{1}{t}, 1-x_1$).
Since $(1-\rho)\frakl_1$ (resp. $(1-\rho)\frakl_t,(1-\rho)\frakl_{1-x_1}$) is an element in 
$H_1(C,\Sigma_1)$, 
$\frakl_1$ (resp. $\frakl_t,\frakl_{1-x_1}$) is an element in 
$\dfrac{1}{1-\rho}H_1(C,\Sigma_1)$
and defines a class $\overline{\frakl_1}$
 (resp. $\overline{\frakl_t},\overline{\frakl_{1-x_1}}$) in
$\big(\dfrac{1}{1-\rho}H_1(C,\Sigma_1)\big)/H_1(C,\Sigma_1)$.
We define $\bold F_3$-vector spaces $Z_{\bold F_3}$ and
$H_{\bold F_3}$ by
\begin{align*}
Z_{\bold F_3}&=\ker([P_0]\bold F_3\oplus [P_1]\bold F_3\oplus [P_t]\bold F_3\oplus [P_\infty]\bold F_3 \to \bold F_3) \\
&
a_0[P_0]+a_1[P_1]+a_t[P_t]+a_{\infty}[P_\infty] \mapsto 
a_0+a_1+a_t+a_{\infty},
\\
H_{\bold F_3}&=
Z_{\bold F_3}/([P_0]-[P_1]-[P_t]+[P_\infty]).
\end{align*}
It is easy to show the following lemma.
\begin{lemma}
\label{lemma:level structure on curve indep}
\begin{enumerate}
\item
The classes $\overline{\frakl_1}, \overline{\frakl_t}$ and $\overline{\frakl_{1-x_1}}$
depend only on 
the end points
$\frakl_1,\frakl_t$ and $\frakl_{1-x_1}$.
\item
The $\bold F_3$-linear map
$\Cal L:\dfrac{1}{1-\rho}H_1(C)/H_1(C)
\to H_{\bold F_3}$
defined by
$$
\Cal L(\overline{\frakl_1})=[P_1]-[P_0],\quad
\Cal L(\overline{\frakl_t})=[P_t]-[P_0]
$$
is an isomorphism independent of
the choice of $\frakl_1$ and $\frakl_t$.
\end{enumerate}
\end{lemma}
We define 
the following classes 
in $\big(\dfrac{1}{1-\rho}H_1(C,\Sigma_1)\big)/H_1(C,\Sigma_1)$:
\begin{align*}
&\big[\dfrac{1}{1-\rho}\beta_1\big]=-\overline{\frakl_1},\quad
\big[\dfrac{1}{1-\rho}\beta_2\big]=\overline{\frakl_t},\quad
\big[\dfrac{1}{1-\rho}\beta_3\big]=\overline{\frakl_{1-x_1}}.
\end{align*}
If $(x_1,x_2)$ belongs to $M$,
then they coincide with the image of 
$\big[\dfrac{1}{1-\rho}\beta_1\big]$,
$\big[\dfrac{1}{1-\rho}\beta_2\big]$ and
$\big[\dfrac{1}{1-\rho}\beta_3\big]$ defined in 
\S \ref{relative homology explicit form}
by the relations
(\ref{homological identity 1})
and (\ref{homological identity 2}).
Using symplectic basis $\{\alpha_1, \alpha_2, \beta_1, \beta_2\}$
the elements $\overline{\frakl_i}$ 
$(i=0,1,t,\infty)$ are written 
as 
$$
\frakl_i\equiv \frac{1}{3}(-p_iU,p_i)\begin{pmatrix}\alpha\\ \beta\end{pmatrix},\quad 
\alpha=\begin{pmatrix}\alpha_1 \\ \alpha_2\end{pmatrix},
\beta=\begin{pmatrix}\beta_1 \\ \beta_2\end{pmatrix},\quad
(i=0,1,t,\infty)
$$
modulo $H_1(C)$, where
\begin{equation}
\label{eq:P-b-ch}
p_0=(0,0),\quad p_1=(1,0),\quad p_t=(0,2),\quad p_\infty=(1,2).
\end{equation}

Next we define the mod $(1-\rho)$ marking of $\widetilde{X}$ using
the marking of $C$.
By choosing a branch of $(1-x_1)^{\frac{1}{3}}(1-x_2)^{\frac{1}{3}}$,
we obtain a rational map
\begin{equation*}
\lambda:E\times C\dasharrow \widetilde{X}
\end{equation*}
by a morphism given in
(\ref{quotient by order 3 action}). By Proposition \ref{prop:invariant in tensor},
the rational map $\lambda$ induces an isomorphism
\begin{align*}
&\bigg(\dfrac{1}{1-\rho}H_1(E)\otimes_{\bold Z[\rho]}
 H_1(C,\Sigma_1)\bigg)
\bigg/
\bigg(H_1(E)\otimes_{\bold Z[\rho]} H_1(C,\Sigma_1)\bigg)
\\
&\xrightarrow[\simeq]{\lambda_*} 
\bigg(H_2(\widetilde{X},E_2)_{\rho}\bigg)\bigg/\bigg((1-\rho)
H_2(\widetilde{X},E_2)_{\rho}\bigg).
\end{align*}
The image of 
$$
\Delta_E\otimes \big(\dfrac{1}{1-\rho}\beta_i\big)=
\delta_E\otimes \beta_i 
$$
under the map $\lambda_*$ is denoted by $\overline{B_i}$.
One can show the following lemma easily.
 \begin{lemma}
  \label{lemma:level structure of K3 surface}
The elements $\overline{B_1},\overline{B_2}$ and $\overline{B_3}$ 
 in
 $\bigg(H_2(\widetilde{X},E_2)_{\rho}\bigg)\bigg/\bigg((1-\rho)
H_2(\widetilde{X},E_2)_{\rho}\bigg)$
do not depend on the choice of 
 $(1-x_1)^{\frac{1}{3}}(1-x_2)^{\frac{1}{3}}$.
 The element $\overline{B_i}$ in
 $\bigg(H_2(\widetilde{X},E_2)_{\rho}\bigg)\bigg/\bigg((1-\rho)
 H_2(\widetilde{X},E_2)_{\rho}\bigg)$
 defined as above is denoted by
 $\overline{B_i}(X)$.
 \end{lemma}

\subsection{Moduli space $\Cal M_{mk}$ of 
marked triple coverings of $\bold P^2$}

\subsubsection{The standard modules and bilinear forms}
\label{subsec:standard module}
We set 
$$
W_{(-2)}=\<B_1, B_2\>_{\bold Z[\rho]}\subset
W_{(-1)}=\<B_1, B_2, B_3\>_{\bold Z[\rho]}.
$$
Here $B_1,B_2,B_3$ form a formal free basis over $\bold Z[\rho]$.
We use an identification 
$
W_{(-1)}  \simeq \bold Z[\rho]^3
$
by writing an element $v$ in $W_{(-1)}$ by
\begin{equation}
\label{base and coefficients}
v=(c_1,c_2,c_3)\begin{pmatrix}
B_1 \\ B_2 \\ B_3 
	       \end{pmatrix},
\quad c_1, c_2, c_3\in \bold Z[\rho].
\end{equation}
Similarly, the submodule $W_{(-2)}$ is 
identified with 
$\{(c_1,c_2) \in \bold Z[\rho]^2\}$.

We define a $\bold Z[\rho]$-valued hermitian form $h(\ ,\ )$ and
a $\bold Z$-valued symmetric bilinear form $\<\ ,\ \>$ on $W_{(-2)}$ by 
\begin{equation}
\label{symmetric form and hermitian form}
 h(x,y)=x U ^t\overline{y},
\quad \<x,y\>=\dfrac{2}{3}\Re(h(x,y)),
\end{equation}
where $U$ is given in (\ref{equ:definition of matrix U}).

\subsubsection{Moduli space of marked of configurations}
 \begin{definition}[Marked configuration]
\label{def:marked config}
We define a marked configuration by
a pair $((x_1,x_2),\mu)$ consisting of
\begin{enumerate}
\item
a point $(x_1,x_2)$ in $\Cal M$,
\item
an  isomorphism (marking) $\mu:W_{(-1)}\to H_2(\widetilde{X},E_2)_{\rho}$ 
of $\bold Z[\rho]$-modules,
\end{enumerate}
satisfying the following three conditions.
\begin{enumerate}[label=(\alph*)]
 \item 
The image of $W_{(-2)}$ is identified with
$H_2(\widetilde{X})_{\rho}$ under the map $\mu$.
Under this isomorphism, 
the symmetric bilinear form on $W_{(-2)}$ and the intersection form
on $H_2(\widetilde{X})_{\rho}$
are compatible.
\item
\label{def of marking property 3}
     Under the map     
$$
W_{(-1)}/W_{(-2)}\to H_2(\widetilde{X},E_2)_{\rho}/
H_2(\widetilde{X})_{\rho}\simeq H_1(E_2)
     $$
     induced by $\mu$,
the element $B_3$
is sent to the classes of $\Delta_E$,
where the second isomorphism is obtained by
the exact sequence 
(\ref{fundamental exact sequence for curves on K3}).
\item
(Level structures) Let
$$
\overline{\mu}:W_{(-1)}/(1-\rho)W_{(-1)} \to 
H_2(\widetilde{X},E_2)_{\rho}/
(1-\rho)H_2(\widetilde{X},E_2)_{\rho}
$$
be the map induced by the map $\mu$.
Then the class of $B_i$ mod $(1-\rho)$ is mapped to the element
     $\overline{B_i}(X)$ defined in Lemma
     \ref{lemma:level structure of K3 surface}.
\end{enumerate}
 \end{definition}
The set of marked configurations is denoted by $\Cal M_{mk}$.

\subsubsection{The case where $\vec\ell\in M$}
A consequence of Proposition 
\ref{prop K3 and ExC transcendental}, we have the following proposition.
\begin{proposition}
Let $\vec\ell=(x_1,x_2)\in M$.
 Using the element $\beta_i$ in $H_1(C)$ defined in
\S \ref{relative homology explicit form},
we define
a $\bold Z[\rho]$-isomorphism $\mu:W_{(-1)}\to H_2(\widetilde{X},E_2)_{\rho}$ by
setting
$$
\mu(B_i)=\lambda(\delta_E\otimes \beta_i).
$$
Then by 
the definition of $\overline{B_i}(X)$ in Lemma \ref{lemma:level structure of K3 surface},
Proposition \ref{hermitian form using delta E beta},
\ref{corresondence of cycles in ExC and X generators}
 and
 the relations (\ref{homological identity 1})
and (\ref{homological identity 2}),
the pair $((x_1,x_2),\mu)$ is a marked configuration.
\end{proposition}

Let $\mu:W_{(-2)}\to H_2(\widetilde{X})_{\rho}$
be a marking in Definition \ref{def:marked config}, 
and $\lambda:E\times C\dasharrow \widetilde{X}$
be a rational map in (\ref{quotient by order 3 action}).
Let 
$\beta_1, \beta_2$
be elements in $H_1(C)$ such that
$$
\mu(B_i)=\lambda(\delta_E\otimes \beta_i).
$$
Then we have the following proposition.
\begin{proposition}
The set $\{
\alpha_1=\rho(\beta_2), \alpha_2=\rho(\beta_1),
\beta_1,\beta_2\}$ forms a symplectic basis. 
Conversely, if $\{\alpha_1,\alpha_2, \beta_1,\beta_2\}$
is a symplectic basis satisfying
\begin{equation}
\label{compatible symplectic base with rho action}
\alpha_1=\rho(\beta_2),\quad
\alpha_2=\rho(\beta_1),
\end{equation}
then the set $\{B_1,B_2\}$ defined by
$B_i=\mu^{-1}(\lambda(\delta_E\otimes \beta_i))$
($i=1,2$) forms a basis of $H_2(\widetilde{X})_{\rho}$
and the intersection form is expressed as 
(\ref{symmetric form and hermitian form}) with respect to this basis.
\end{proposition}

\subsection{Period integrals of marked K3 surfaces}
In this section, we define a period map from 
$\Cal M_{mk}$ to $\bold B\times \bold C^2$ using the mixed Hodge
structure of $H_2(X,E_2)_{\rho}$.
Let
$\mu:W_{(-1)} \to H_2(\widetilde{X},E_2)_{\rho}$
be a marking of $\widetilde{X}$.

We consider the following commutative diagrams
for de Rham cohomologies:
$$
\begin{matrix}
0\to &H_{dR}^1(E_2)(\chi) &\to &H_{dR,c}^2(\widetilde{X}-E_2)(\chi) &\to &
H_{dR}^2(\widetilde{X})(\chi)
&\to 0
\\
& & & \cup & & \cup \\
& & & F^2H_{dR,c}^2(\widetilde{X}-E_2)(\chi)
 & \xrightarrow[\pi]{\simeq} &
 F^2H_{dR}^2(\widetilde{X})(\chi)
\end{matrix}
$$
$$
\begin{matrix}
0\to &H_{dR}^1(E_2)(\overline{\chi})  &\to &
H_{dR,c}^2(\widetilde{X}-E_2)(\overline{\chi}) 
&\to &H_{dR}^2(\widetilde{X})(\overline{\chi})
&\to 0
\\
& & & \cup & & \cup \\
& & & F^1H_{dR,c}^2(\widetilde{X}-E_2)(\overline{\chi})
 & \xrightarrow[\pi']{\simeq} &
 F^1H_{dR}^2(\widetilde{X})(\overline{\chi}).
\end{matrix}
$$
\begin{proposition}
The spaces 
$F^2H_{dR}^2(\widetilde{X})(\chi)$ and
$F^1H_{dR}^2(\widetilde{X})(\overline{\chi})$
are one dimensional over $\bold C$.
\end{proposition}
\begin{proof}
Using the isomorphism (\ref{isom for de Rham cohomology}), 
we have 
\begin{align*}
& F^2H_{dR}^2(\widetilde{X})(\chi)=
\lambda^*(H^{10}(E)(\chi)\otimes H^{10}(C)(\chi)),
\\
& F^1H_{dR}^2(\widetilde{X})(\overline{\chi})=
\lambda^*(H^{01}(E)(\overline{\chi})\otimes H^{10}(C)(\overline{\chi})). 
\end{align*}
Thus we have the proposition.
\end{proof}
By the above proposition, the spaces
$F^2H_{dR}^2(\widetilde{X})(\chi)$ and
$F^1H_{dR}^2(\widetilde{X})(\overline{\chi})$ can be regarded as
one dimensional
subspaces in
\begin{align*}
&(W_{(-1)}\otimes \bold C)(\chi)^*
\xleftarrow{\mu}
H_{dR,c}^2(\widetilde{X}-E_2)(\chi)
\text{ and } 
\\
&(W_{(-1)}\otimes \bold C)(\overline{\chi})^*
\xleftarrow{\mu}
H_{dR,c}^2(\widetilde{X}-E_2)(\overline{\chi}).
 \end{align*}
This one dimensional vector spaces are 
expressed by a matrix $P((x_1,x_2),\mu)$ in
the following way.
Let $\xi$ and $\xi'$ be bases of
$F^2H^2(\widetilde{X})(\chi)$ and
$F^1H^2(\widetilde{X})(\overline{\chi})$.
Via the isomorphism $\pi$ and $\pi'$ in the above
diagrams,
$\xi$ and $\xi'$
are regarded as elements
in $H^2_{dR,c}(\widetilde{X}-E_2)_{\rho}$.
Using
$\mu(B_1), \mu(B_2), \mu(B_3) \in H_2(\widetilde{X},E_2)_{\rho}$,
we define the period matrix $P((x_1,x_2),\mu)$ of a 
marked configuration $((x_1,x_2),\mu)$
as follows
\begin{align*}
P((x_1,x_2),\mu)=
\begin{pmatrix}
\<\mu(B_1),\xi\> &
\<\mu(B_1),\xi'\>
\\
\<\mu(B_2),\xi\> &
\<\mu(B_2),\xi'\>
\\
\<\mu(B_3),\xi\> &
\<\mu(B_3),\xi'\>
\end{pmatrix}.
\end{align*}
Here
$\<*,*\>$ is the pairing between the relative homology and the 
de Rham cohomology.

 \begin{proposition}
\label{eta is in B}
Let $((x_1,x_2),\mu)$
be an element in $\Cal M$ and
$P((x_1,x_2),\mu)$ be the period matrix of $((x_1,x_2),\mu)$ defined as above.
Then 
\begin{equation}
\label{K3 symetricity}
\dfrac{\<\mu(B_1),\xi\>}{
\<\mu(B_2),\xi\>}
=-\dfrac{\<\mu(B_1),\xi'\>}{
\<\mu(B_2),\xi'\>},
\end{equation}
and
$\eta=\dfrac{\<\mu(B_1),\xi\>}{
\<\mu(B_2),\xi\>}$
is an element of $\bold B$.
The elements $\eta$
is independent of the choice of $\xi$ and $\xi'$.
 \end{proposition}
\begin{proof}
The spaces 
$H_{dR}^2(\widetilde{X})(\chi)$ and 
$H_{dR}^2(\widetilde{X})(\overline{\chi})$ is identified with 
$\bold C^2$ via the 
following isomorphisms:
\begin{align*}
&\varphi:H_{dR}^2(\widetilde{X})(\chi)\ni a
\mapsto 
\varphi(a)=\begin{pmatrix}
\<\mu(B_1),a\> 
\\
\<\mu(B_2),a\> 
\end{pmatrix}
\in \bold C^2,
\\
&\varphi':H_{dR}^2(\widetilde{X})(\overline{\chi})\ni a'
\mapsto 
\varphi(a)=\begin{pmatrix}
\<\mu(B_1),a'\> 
\\
\<\mu(B_2),a'\> 
\end{pmatrix}
\in \bold C^2.
\end{align*}
Let $\pi_{\chi}$ (resp. $\pi_{\overline{\chi}}$) be the image of the projection to the 
$\chi$-part 
(resp. $\overline{\chi}$-part)
according to the direct sum decomposition:
$$
H_2(\widetilde{X})_{\rho}\otimes \bold C \simeq T_X^*\otimes \bold C\simeq 
H_{dR}^2(\widetilde{X})(\chi) \oplus
H_{dR}^2(\widetilde{X})(\chi).
$$
Then we have
\begin{align*}
&\varphi(a_1\pi_{\chi}(B_1)+a_2\pi_{\chi}(B_2))=\dfrac{1}{3}(a_2,a_1), \\
&\varphi'(a_1\pi_{\overline{\chi}}(B_1)+a_2\pi_{\overline{\chi}}(B_2))=\dfrac{1}{3}(a_2,a_1).
\end{align*}
Since the cup product is given by the formula
(\ref{symmetric form and hermitian form}),
we have
$$
a\cup a'=3\cdot ^t\varphi(a)U\varphi'(a')
$$
for $a\in H_{dR}^2(\widetilde{X})(\chi), a'
\in H_{dR}^2(\widetilde{X})(\overline{\chi})$.
Since $\xi$ and $\xi'$ are contained in 
$F^2H^2_{dR}(X)(\chi)$ and
$F^1H^2_{dR}(X)(\overline{\chi})$, we have $\xi\cup \xi'=0$.
Thus equality (\ref{K3 symetricity}) follows.

The complex conjugate 
$H^2_{dR}(\widetilde{X})(\chi) \to H^2_{dR}(\widetilde{X})(\overline{\chi})$
with respect to $T_X^*$ is given by $\ ^t(a_1,a_2)\to \ ^t(\overline{a_1},\overline{a_2})$
via the identification $\varphi$ and $\varphi'$. By the positivity of the polarization, we have 
$\xi\cup \overline{\xi}>0$
for a nonzero element $\xi\in F^2H_{dR}^2(\widetilde{X})(\chi)$.
Therefore $\Re(\eta)>0$.
\end{proof}
 \begin{definition}
\label{modular embedding}
We define the period domain $\Cal D$ by
$$
\Cal D=\bold B\times \bold C^2,
$$
and $per:\Cal M_{mk} \to \Cal D$
by
$$
per((x_1,x_2),\mu)=(
\eta,z)\in \Cal D=\bold B\times \bold C^2.
$$
Here, $\eta$ is defined in Proposition \ref{eta is in B}
and $z$ is defined by
$$
z=
\left(
 \frac{\<\mu(B_3),\xi\>}{\<\mu(B_2),\xi\>},
 \frac{\<\mu(B_3),\xi'\>}{\<\mu(B_2),\xi'\>}
\right). 
$$
The vector $z$ is also independent of the choice of $\xi$ and $\xi'$.
 \end{definition}

\subsection{Transport of markings and a group action}
\subsubsection{Definition of $G(1-\rho)$ and its action on $W_{(-1)}$}
\label{definition of dis gp}
We define subgroups $\Gamma$ and $\Gamma(1-\rho)$ 
of $GL(2,\bold Z[\rho])$
by 
\begin{align*}
\Gamma&=\{g \in GL(2, \bold Z[\rho])\mid
g U \ ^t\overline{g}= U\}, \\
\Gamma(1-\rho)&=\{g \in \Gamma \mid
g \equiv I_2 \text{ mod }(1-\rho)\}, 
\end{align*}
and a subgroup 
$G$ and
$G(1-\rho)$ of 
$GL(3,\bold Z[\rho])$ by
\begin{align*}
\label{equ:def of G(1-rho)}
G&=\{\widetilde{g}=\begin{pmatrix}g & 0 \\ b & 1
\end{pmatrix} \in GL(3,\bold Z[\rho])\mid \ 
g\in \Gamma\},
\\
\nonumber
G(1-\rho)&=\{\widetilde{g}\in G
\mid \ 
\widetilde{g}\equiv id \text{ mod }(1-\rho) \}.
\end{align*}
The group $GL(3,\bold Z[\rho])$ acts on 
$W_{(-1)}$
from the right via the expression (\ref{base and coefficients}).

\subsubsection{The action of $G(1-\rho)$ on $\Cal M_{mk}$}
Let $((x_1,x_2),\mu)$ be a marked configuration and
$\widetilde{g}$ an element in 
$G(1-\rho)$.
By taking the composite $\mu\circ \widetilde{g}$
of $\widetilde{g}$ and
the marking 
$\mu:W_{(-1)}\to H_2(X,E_2)_{\rho}$,
we get an action of  $G(1-\rho)$ on $\Cal M_{mk}$.

By the expression of the intersection form of
the generic transcendental lattice of $\widetilde{X}$ obtained in
Proposition \ref{hermitian form using delta E beta},
the action of $G(1-\rho)$ preserves the intersection
form $\<\ ,\ \>_X$ on $H_2(\widetilde{X})$.
As a consequence, the group $G(1-\rho)$
acts on the moduli space $\Cal M_{mk}$ of marked configurations.
 \begin{proposition}
\label{quotient and iso}
The quotient of $\Cal M_{mk}$ by $G(1-\rho)$
is isomorphic to $\Cal M$.
\end{proposition}
 \begin{proof}
The natural map $\Cal M_{mk}\to \Cal M$ is surjective by
definition of $\Cal M_{mk}$.
We show that the fiber $\Cal M_{mk} \to \Cal M$
is transitive under the action of $G(\rho-1)$.
Let $(x_1, x_2)$ be an element in $\Cal M$, and
$((x_1,x_2),\mu)$ and
$((x_1,x_2),\mu')$ be
two marked configurations.
Let $\widetilde{g}={\mu'}^{-1}\circ \mu$ be the composite map
$$
{\mu'}^{-1}\circ \mu:
W_{(-1)} 
\xrightarrow{\mu} H_2(\widetilde{X},E_2)_{\rho}
\xrightarrow{{\mu'}^{-1}} W_{(-1)}.
$$
Then $\widetilde{g}$ becomes an automorphism of $W_{(-1)}$ compatible
with the $\rho$ action on $W_{(-1)}$.
Since the submodule $W_{(-2)}$ is mapped isomorphically to
the subspace $H_2(\widetilde{X})_{\rho}$
under the isomorphisms $\mu$ and $\mu'$,
the submodule $W_{(-2)}$ is stable under the isomorphism $\widetilde{g}$.
Let $g$ be the restriction of $\widetilde{g}$ to the submodule $W_{(2)}$.
Under the identifications $\mu$ and $\mu'$ of $W_{(2)}$ with $H_2(\widetilde{X})_{\rho}$,
intersection forms on $H_2(\widetilde{X})_{\rho}$ is transformed
the inner product on $W_{(-2)}$. Since $g$ preserves the action of $\rho$, 
the hermitian form $h$ is preserved by $g$.
Moreover by the condition for level structures for $\mu$ and $\mu'$,
$\overline B_1,\overline B_2$ and $\overline B_3$ are mapped to
$\overline B_1(X),\overline B_2(X)$ and $\overline B_3(X)$ by the isomorphisms
$$
\overline{\mu}, \overline{\mu'}:
W_{(-1)}/(1-\rho)W_{(-1)}\simeq H_2(\widetilde{X})_{\rho}/(1-\rho)H_2(\widetilde{X})_{\rho}.
$$
Therefore we have $g\equiv I_3$ mod $(1-\rho)$.
As a consequence, $g$ is an element in $G(1-\rho)$.
 \end{proof}

\subsubsection{The action of $G(1-\rho)$ on $\Cal D$}
The characters $\chi$ and $\overline{\chi}$
induce ring homomorphisms $\bold Z[\rho]\to
\bold C$ and they induce
group homomorphisms
\linebreak
$GL(3,\bold Z[\rho]) \to GL(3,\bold C)$,
which are also denoted by $\chi$ and
$\overline{\chi}$.

Let $B_1, B_2, B_3$ be
the $\bold Z[\rho]$-basis of
$W_{(-1)}$ defined in \S \ref{subsec:standard module}.
Using this basis, the group $G(1-\rho)$ acts on $W_{(-1)}$
via the identification given in (\ref{base and coefficients}).
Thus an element $g$ in $G(1-\rho)$ acts on the set of the pairs of one dimensional vector spaces
$$
\bigg\{(F^2,F^1)\ \bigg|\ 
\begin{matrix}
F^2 \subset
(W_{(-1)}\otimes \bold C)(\chi)^*,\ 
F^1\subset
(W_{(-1)}\otimes \bold C)(\overline{\chi})^* ,\\
\dim (F^2)=\dim(F^1)=1
\end{matrix}
\bigg\}.
$$
This action 
is expressed
as
\begin{align*}
\chi(\widetilde{g}) \times \overline{\chi}(\widetilde{g}):
(\xi \bold C,
\overline{\xi}\bold C)
\mapsto
(\chi(\widetilde{g})
\xi \bold C,
\overline{\chi}(\widetilde{g})
\overline{\xi}\bold C).
\end{align*}
By transporting the structure, 
we have an action of $G(1-\rho)$ on $\Cal D$, which is written as
$$
(\eta,z)=(\eta,z_1,z_2)\mapsto
(\frac{g_{11}\eta+g_{12}}{g_{21}\eta+g_{22}},
\frac{z_1+w_1\eta+w_2}{g_{21}\eta+g_{22}},
\frac{z_2-\overline{w_1}\eta+\overline{w_2}}{-\overline{g_{21}}\eta+
\overline{g_{22}}}),
$$
for
$$
\widetilde{g}=\begin{pmatrix}
g_{11} & g_{12} & 0 \\
g_{21} & g_{22} & 0 \\
w_1 & w_2 &1
\end{pmatrix}
\in  G(1-\rho).
$$
As a consequence, we have the following proposition.
\begin{proposition}
The above action defines an associative action.
Moreover the map $\Cal M_{mk}\to \Cal D$ 
is equivariant under
the action of $G(1-\rho)$.
By Proposition \ref{quotient and iso}, we have a map
\begin{equation}
\label{induced map on quotient} 
\Cal M \to \Cal D/G(1-\rho).
\end{equation}
\end{proposition}

\subsection{Existence of an extra involution}
Let $t$ be an element in $\bold C-\{0,1\}$, and
$C$ be a curve defined by the equation
(\ref{def eq of ellipic curve with w-action}).

\begin{proposition}
\label{extra involution for general case}
\begin{enumerate}
\item
There exists a symplectic basis 
$\{\alpha_1, \alpha_2, \beta_1, \beta_2\}$
and an involution $\iota$ of $C$
satisfying the following properties.

\begin{enumerate}
\item
The relation (\ref{compatible symplectic base with rho action}) holds.
\item
The classes 
$\big[\dfrac{1}{1-\rho}\beta_1\big]$ and 
$\big[\dfrac{1}{1-\rho}\beta_2\big]$
in $\dfrac{1}{1-\rho}H_1(C)/H_1(C)$ are equal to  
$[P_0]-[P_1]$ and $[P_t]-[P_0]$ under the isomorphism $\Cal L$ defined in 
Lemma \ref{lemma:level structure on curve indep}.
\item
The following equalities hold:
\begin{equation}
\label{good base for extra involution}
\iota(\beta_1)=\beta_1,\quad
\iota(\beta_2)=-\beta_2,\quad
\iota\circ \rho=\rho^{-1}\circ\iota.
\end{equation}
\end{enumerate}
\item
Let $\{\alpha_1, \alpha_2, \beta_1, \beta_2\}$
be a symplectic basis satisfying the conditions (a), (b) of (1).
Then there exists a unique involution $\iota$ of $C$ satisfying the
equalities in (\ref{good base for extra involution}).
\end{enumerate}
 \end{proposition}

We define an extra involution of $C$ for a general $t\in \bold C^\{0,1\}$.
\begin{definition}
The involution $\iota$ satisfying these equalities
is called the extra involution of $C$ 
with respect to the symplectic basis 
$\{\alpha_1, \alpha_2, \beta_1, \beta_2\}$.
\end{definition}
\begin{proof}
(1) In \S \ref{sec:extra involution}, 
the existence of the above basis and the involution $\iota$
is proved in the case $(x_1^0,x_2^0)\in M$.
For a general element $(x_1^1, x_2^1)$ in $\Cal M$, we choose a path 
starting from $(x_1^0, x_2^0)$
and ending with a point $(x_1^1, x_2^1) \in M$.
The values of $t$ at $(x_1^0, x_2^0)$ and $(x_1^1, x_2^1)$ 
are denoted by $t_0$ and $t_1$.
Let $\beta_1^0, \beta_2^0$ be topological cycles on $C_{t_0}$
defined in \S \ref{relative homology explicit form}.
Since the parameter $t$ varies continuously on $(x_1,x_2)$,
$\beta_1^0, \beta_2^0$ are deformed continuously and
we get cycles $\beta_1^1, \beta_2^1$ in
$H_1(C_{t_1})$, which satisfy the properties (a), (b).
The involution $\iota_0$ of
the curve $C_{t_0}$ 
induces an involution $\overline{\iota_0}$ of $\bold P^1$
satisfying $\overline{\iota_0}(P_0)=P_1$ and 
$\overline{\iota_0}(P_t)=P_{\infty}$.
We have a deformation $\overline{\iota_t}$ of the involution of $\bold P^1$
transposing $\{P_0, P_1\}$ and $\{P_t, P_{\infty}\}$
and get an involution $\overline{\iota_1}$ on $\bold P^1$ for $t=t_1$.
We can lift the involution $\overline{\iota_t}$ of $\bold P^1$
to an involution $\iota_t$ of $C_t$ which is a deformation of $\iota_0$.
Since $\iota_t$ is continuous on $t$,
the equality (c) is preserved under this deformation.

(2)
Let $\{\beta_1,\beta_2\}$ be a $\bold Z[\rho]$-basis of 
$H_1(C)$ satisfying the condition (a), (b).
We choose $\bold Z[\rho]$-basis 
$\{\beta'_1,\beta'_2\}$ of $H_1(C)$, and an involution $\iota$
of $C$ satisfying the conditions (a), (b) and (c) of (1).
Then we have
$$
\begin{pmatrix}
 \beta_1' \\ \beta_2'
\end{pmatrix}
=g\begin{pmatrix}
 \beta_1 \\ \beta_2
\end{pmatrix},\quad 
g=\begin{pmatrix}
a & b \\ c & d   
  \end{pmatrix}
\in \Gamma(1-\rho),
$$
where $\Gamma(1-\rho)$ is defined in \S \ref{definition of dis gp}.
We set 
$\displaystyle
D=\begin{pmatrix}
  1 & 0 \\ 0 & -1 
  \end{pmatrix}.
$
Then we have
\begin{align*}
\begin{pmatrix}
\iota
(\beta_1) \\ \iota(\beta_2)
\end{pmatrix}
&=\overline{g}\begin{pmatrix}
 \iota(\beta_1') \\ \iota(\beta_2')
\end{pmatrix} 
=\overline{g}D\begin{pmatrix}
\beta_1' \\ \beta_2'
\end{pmatrix}
=\overline{g}DU^t\overline{g}U
\begin{pmatrix}
\beta_1 \\ \beta_2
\end{pmatrix}\\
&=
\overline{ad-bc}\cdot
D
\begin{pmatrix}
\beta_1 \\ \beta_2
\end{pmatrix}=
\overline{\det(g)}\cdot
D
\begin{pmatrix}
\beta_1 \\ \beta_2
\end{pmatrix}.
\end{align*}
Since $g\in \Gamma(1-\rho)$, $\det(g)$ is a unit of 
$\bold Z[\rho]$
and congruent to $1$ mod $(1-\rho)$. Thus there exists $p\in\{0,1,2\}$
such that $\overline{\det(g)}=\rho^{p}$.
By setting $\iota'=\rho^p\iota \rho^{-p}$, we have
$\iota'\circ \rho=\rho^{-1}\circ \iota'$ and
\begin{align*}
\begin{pmatrix}
\iota'(\beta_1) \\ \iota'(\beta_2)
\end{pmatrix}
=\begin{pmatrix}
\rho^p\iota\rho^{-p}(\beta_1) \\ 
\rho^p\iota\rho^{-p}(\beta_2)
\end{pmatrix}
=\begin{pmatrix}
\rho^{-p}\iota(\beta_1) \\ 
\rho^{-p}\iota(\beta_2)
\end{pmatrix}
=\rho^{-p}
\overline{\det(g)}D
\begin{pmatrix}
 \iota(\beta_1) \\ \iota(\beta_2)
\end{pmatrix} 
=D
\begin{pmatrix}
 \iota(\beta_1) \\ \iota(\beta_2)
\end{pmatrix}.
\end{align*}
Therefore $\iota'$ satisfies the condition (c).
The uniqueness of $\iota$ can be proved similarly.
\end{proof}
\begin{definition}
Let 
$\{\alpha_1,\alpha_2,\beta_1,\beta_2\}$
be a symplectic basis satisfying the condition (a), (b) of
Proposition \ref{extra involution for general case} (1).
The involution $\iota$ satisfying the condition (c)
of Proposition \ref{extra involution for general case}(1) 
is called the extra involution associated to
the basis $\{\alpha_1,\alpha_2,\beta_1,\beta_2\}$.
The element $\iota_{C*}\beta_3\in H_1(C,\Sigma_2)$ is denoted by $\beta_4$.
\end{definition}
Let $((x_1,x_2),\mu)$ be an element in $\Cal M_{mk}$.
By choosing a branch of $(1-x_1)^{\frac{1}{3}}(1-x_2)^{\frac{1}{3}}$,
we obtain a rational map $\lambda:E\times C\dasharrow \widetilde{X}$
as in (\ref{quotient by order 3 action})
and elements $\beta_1, \beta_2, \beta_3$ in $H_1(C,\Sigma_1)$
such that
$$
\mu(B_i)=\lambda_*(\delta_E\otimes \beta_i).
$$

\subsection{Coincidence of period maps for K3 surfaces and
monodromy curves}
\label{sec:coincid per map}
\subsubsection{The case $(x_1,x_2)\in \Cal M$}
We compute the period matrix $P((x_1,x_2),\mu)$
for $(x_1,x_2)\in \Cal M$
using the period integrals of the curve $C$.
Let $\lambda:E\times C\dasharrow \widetilde{X}$ be a rational map in
(\ref{quotient by order 3 action}),
and $\iota_C$ be the extra involution associated to $\mu$
and $\lambda$.
Let $\psi_1$ be a non-zero element of $H^0(C,\it{\Omega}^1_C)(\chi)$
and set 
$$
\psi_2=\iota_C^*(\psi_1).
$$
Then we have 
$\psi_2\in H^0(C,\it{\Omega}_C^1)(\overline{\chi})$.
We choose elements $\Omega_E$ and $\Omega'_E$ in 
$H^1(E)(\chi)$ and
$H^1(E)(\overline{\chi})$
such that
$$\lambda^*\xi=\Omega_E\wedge \psi_1,\quad 
\lambda^*\overline{\xi}=\Omega'_E\wedge \psi_2.
$$
We choose elements $\beta_1, \beta_2, \beta_3$ in $H_1(C,\Sigma_1)$
such that
$$
\mu(B_i)=\lambda_*(\delta_E\otimes \beta_i).
$$
By setting
$$
c_E=\int_{\delta_E}\Omega_E,\quad 
\overline{c_E}=\int_{\delta_E}\Omega_E',
$$
we have
\begin{align*}
&\<\mu(B_i),\xi\>
=\<\lambda_*(\delta_E\otimes \beta_i),
\lambda^*(\xi)\>
=\<\delta_E\otimes \beta_i,
\Omega_E\wedge \psi_1\>
=
c_E \int_{\beta_i}\psi_1,
\\
&\<\mu(B_i),\overline{\xi}\>
=\<\delta_E\otimes \beta_i,
\Omega'_E\wedge \psi_2\>
=
\overline{c_E} \int_{\beta_i}\psi_2
=
\overline{c_E} \int_{\beta_i}\iota_C^*\psi_1
=
\overline{c_E} \int_{\iota_{C*}\beta_i}\psi_1.
\end{align*}
Since 
$$
\iota_{C*}\beta_1=\beta_1,\quad
\iota_{C*}\beta_2=-\beta_2,\quad
\iota_{C*}\beta_3=\beta_4,
$$
we have
\begin{align}
\label{period mantrix using y's}
P((x_1,x_2),\mu)=&
\begin{pmatrix}
y_1 &
y_1 
\\
y_2 &
-y_2 
\\
y_3 &
y_4 
\end{pmatrix}
\begin{pmatrix}
c_E& 0
\\
0&\overline{c_E}
\\
\end{pmatrix},\quad y_i=\int_{\beta_i}\psi_1\quad(i=1,2,3,4).
\end{align}
Therefore by Definition \ref{modular embedding},
we have
\begin{equation}
\label{relation between per of K3 and mon curve}
per((x_1,x_2),\mu))=(\dfrac{y_1}{y_2},\dfrac{y_3}{y_2},\dfrac{y_4}{y_2}).
\end{equation}
\begin{definition}
We define a map
$$
\jmath_{\Cal D}:\Cal D 
\ni (\eta,z
)=(\eta,z_1,z_2) \mapsto  (\tau(\eta), \zeta(z)) \in
\bold H_2\times \bold C^2
$$
by 
\begin{align*}
&\tau=\frac{1}{2}\begin{pmatrix}
\sqrt{-3}\eta^{-1}  & -1 \\
-1 & \sqrt{-3}\eta 
\\
\end{pmatrix},\\
&\zeta=
\frac{1}{2}\left(
(\dfrac{z_1}{1-\omega}-\dfrac{z_2}{1-\omega^2})\eta^{-1},
\dfrac{z_1}{1-\omega}+\dfrac{z_2}{1-\omega^2}
\right). 
\end{align*}
\end{definition}
By the equality (\ref{relation between per of K3 and mon curve}),
$\tau$ is the normalized period matrix of the curve $C$ 
for the symplectic base
$\{\alpha_1,\alpha_2,\beta_1,\beta_2\}$.
We define the Jacobian $J(C)$ of the curve $C$ by
$$
J(C)=\bold C^2/(\bold Z^2\tau \oplus\bold Z^2).
$$
The class of $\zeta$ in
the Jacobian $J(C)$ is equal to the 
image of Abel-Jacobi map for the integral on the path
$\frakl_{1-x_1}$ defined in \S \ref{subsec:level strucutre on K3}.

\subsubsection{The relation between $P((x_1,x_2),\mu)$ and
integrals $\varphi_i$ for $(x_1,x_2)\in M$}
\label{explicit markings and periods}
We give explicit computations of the period
matrix in the case $(x_1,x_2) \in M$ given in
(\ref{real domain moduli of K3}). 
We choose the marking $\mu$ by setting
$\mu(B_i)=B_i^*$ for $i=1,2,3$, where $B_i^*$ are defined in 
Proposition \ref{corresondence of cycles in ExC and X generators}.
By this choice of $\mu$, the period matrix $P((x_1,x_2),\mu)$ in 
(\ref{period mantrix using y's})
is computed from the integrals
$\varphi_1, \varphi_2, \varphi_3, \varphi_4$ in 
(\ref{integration by stnadard branch})
as follows.
We choose $\Omega_E$ and $\Omega'_E$ so that $c_E=\overline{c_E}$.
By the relation in Proposition \ref{def:marked-configuration}, 
$y_1, y_2, y_3, y_4$ in the left hand side of (\ref{period mantrix using y's})
are given by
\begin{align*}
c_Ey_1&=\omega(1-\omega^2)\varphi_2,\quad
c_Ey_2=(1-\omega^2)\varphi_1
+(1-\omega^2)\varphi_2,\\  
c_Ey_3&=(1-\omega)\varphi_3,\quad
c_Ey_4=(1-\omega^2)\varphi_4
-\omega^2(1-\omega^2)\varphi_2.
\end{align*}

\subsection{Modular embedding}
\label{def:tau from eta}
We set
\begin{align*}
Sp_2(\mathbf{Z})&=\Big\{g\in GL_4(\mathbf{Z})\Big|
\;^tgJ g=J=\begin{pmatrix} & -I \\ I &\end{pmatrix}\Big\},\\
\bold H_2& =\{\tau \in GL_2(\bold C)\mid \ ^t\tau=\tau, \quad \Im(\tau)>0\}.
\end{align*}
We introduce a symplectic form on 
$\bold Z[\rho]\beta_1\oplus \bold Z[\rho]\beta_2$
by (\ref{rho invariant symplectic form}).
Then $\alpha_1=\rho(\beta_2),\alpha_2=\rho (\beta_1),
\beta_1, \beta_2$ form a symplectic basis. 
Using this basis,
we have inclusions
$\jmath_{\Gamma}:\Gamma \to Sp_2(\bold Z)$ and
$\jmath_{G}:G(1-\rho) \to Sp_2(\bold Z)\ltimes \bold Z^4$.
More concretely, they are written as
\begin{align*}
&\jmath_H:\bold Z[\rho]^2\ni (r_1+r_2\rho,s_1+s_2\rho)=(s_2, r_2, r_1, s_1)\in
\bold Z^4,
\\
&\jmath_{\Gamma}:\Gamma\ni g=g_1+g_2\rho
\mapsto 
\begin{pmatrix}
U(g_1-g_2)U & -Ug_2\\
g_2U & g_1
\end{pmatrix}\in Sp_2(\mathbf{Z}),
\\
&\jmath_{G}:G(1-\rho)\ni 
g=
\begin{pmatrix}
g & 0\\
b & 1
\end{pmatrix}\
\mapsto 
\jmath_{\Gamma}(g)\ltimes \jmath_H(b)
\in Sp_2(\mathbf{Z}) \ltimes \bold Z^4.
\end{align*}
By the construction, and the expression of $\zeta$ in
(\ref{def of zeta}),
we have the following lemma.
\begin{lemma}
The inclusion $\jmath_{\Cal D}$ is compatible with the action of
$G(1-\rho)$through $\jmath_G$
That is, we have
$$\jmath_{\Cal D}(g\cdot x)=\jmath_{G}(g)\cdot 
\jmath_{\Cal D}(x),\quad 
x\in \Cal D,\ g\in G(1-\rho).$$

\end{lemma}

\section{Theta values and parameters of configurations}
\label{sec:theta-inv}
In this section, we show that the image of the
period map \linebreak
$per:\Cal M_{mk} \to \Cal D$
coincides with the zero locus $V(\vartheta)$ of 
a theta function $\vartheta
=\vartheta_{(-2,-1,1,2)/3}
(\jmath_{\Cal D}(\eta,z))$.
Moreover, the map (\ref{induced map on quotient}) 
induces an isomorphism
$\overline{per}:\Cal M \to V(\vartheta)/G(1-\rho)$.
We construct the inverse 
of $\overline{per}$
using modular embedding defined in Definition \ref{modular embedding}
and theta functions
(Theorem \ref{th:x1,x2-exp}).

\subsection{Theta function of the Jacobian $J(C)$ and inverse period map}
The theta function $\vartheta_{a,b}$ of $(\tau,\zeta)\in 
\mathbf{H}_2\times \mathbf{C}^2$ with characteristics $(a,b)$ ($a,b \in
\bold Q^2$) is defined by 
$$\vartheta_{a,b}(\tau,\zeta)=\sum_{n\in \mathbf{Z}^2}
\exp[\pi\sqrt{-1}\{(n+a)\tau\;^t(n+a)+2(n+a)\;^t(\zeta+b)\}].$$

In this section we study period map and its inverse using theta
functions.
The image of the map $\Cal M_{mk}\to \Cal D$ is characterized
by the following theorem. 
\begin{definition}
We define an analytic space $V(\vartheta)$
of $\Cal D$ by
$$
V(\vartheta)=\{
(\eta,z)\in \Cal D\mid 
\vartheta_{(-2,-1,1,2)/3}
(\jmath_{\Cal D}(\eta,z))
=0
\}.
$$ 
\end{definition}
\begin{theorem}
\label{th:image}
The image of $per:\Cal M_{mk} \to \Cal D$ coincides with $V(\vartheta)$.
\end{theorem}

We fix a real number $t$ satisfying $0<t<1$.
Let
$C=C_t$ be the curve defined in
(\ref{def eq of ellipic curve with w-action}).
We choose $x_1 \in \bold R$ such that $0<x_1<1$.
Then $x_2$ is determined by the equality 
(\ref{definition of t}).
We use the symplectic basis in \S \ref{relative homology explicit form}.

Let $p=(u,w)$ be the point on the second sheet in $C$ 
with $u=1-x_1$, and $\delta_3$ be a path
from $P_0$ to $p$ on this sheet.
Then $\zeta\in \mathbf{C}^2$ defined  
in (\ref{def of zeta})
 is a vector-valued 
function on $x_1$.
This function is continued analytically to 
a holomorphic function 
$\zeta(\widetilde{p})$ on 
the universal covering $\widetilde{C}$ of $C$.
Since $\zeta(\widetilde{p})$ is a multivalued function on
$p$ in $C$, the function
$$
\vartheta_{a,b}(p)=\vartheta_{a,b}(\tau(\eta),\zeta(p))
$$ 
is a multivalued function on $p$.
By the quasi periodicity of the theta function, 
a zero point of this function and its order are well defined
on $J(C)$.
Riemann's theorem states that if 
$\vartheta_{a,b}(p)$ is not 
identically zero then it has two zero points 
with counting multiplicity.

\begin{lemma}
\label{lem:rho-act}
For $m\in \mathbf{Z}^2$, the order of zero of 
$$\vartheta_{[m]}(p)=\vartheta_{-mU/3,m/3}(p)$$ 
at $p=P_i$ $(i=0,1,t,\infty)$ is congruent to 
$$(-1)^{r_i} (m+p_i) U \;^t  (m+p_i)\bmod 3.$$ 
Here $P_i\in C$ and $p_i\in \mathbf{Z}^2$ are given in 
(\ref{eq:ramification-pts})
and 
(\ref{eq:P-b-ch}), 
respectively, and
$$r_0=1,\quad r_1=2,\quad r_t=2,\quad r_\infty=1.$$
\end{lemma}
\proof
We use the following fundamental property:
if a holomorphic function $f$ around $z=0$ satisfies 
$$\lim_{z\to 0}\frac{f(\omega  z)}{f(z)}=\omega ^k\quad (k=0,1,2)$$
then the order of zero of $f(z)$ at $z=0$ is congruent to $k$ 
modulo $3$. 

We consider the pull back $\rho^*(\vartheta_{[m]})(p)$ of $\vartheta_{[m]}(p)$ 
under the covering transformation $\rho$. 
One can choose a local parameter $v$ 
of $P_0$ (resp. $P_1$, $P_t$ and $P_{\infty}$)
on a neighborhood $U_0$ 
(resp. $U_1$, $U_t$ and $U_{\infty}$) 
such that $\rho^*v=\omega v$ (resp. 
$\rho^*v=\omega^2v$,
$\rho^*v=\omega^2v$ and $\rho^*v=\omega v$).
For a point $p$ in the neighborhood $U_i$, we choose a path 
from $P_i$ to $p$ in $U_i$ and define a function
$$\vartheta_{[m]}^i(p)=
\vartheta_{-(m+p_i)U/3,(m+p_i)/3}
\left(\tau,\int_{P_i}^{p}\Psi\right),\quad 
\Psi=(\psi_1,\psi_2)(\tau_B)^{-1}
$$
on $U_i$, where $\psi_1, \psi_2$
is defined in \S \ref{relative homology explicit form}.
Then we have
$\vartheta_{[m]}(p)=h(p)\vartheta_{[m]}^i(p)$ for a non-zero
holomorphic function $h(p)$.

We compute the limit 
$\lim\limits_{p\to P_i}\dfrac{\rho^*(\vartheta_{[m]}^i(p))}
{\vartheta_{[m]}^i(p)}$
as follows. We have 
\begin{eqnarray*}
\frac{\rho^*(\vartheta_{[m]}^i(p))}{\vartheta_{[m]}^i(p)} &=&
\frac{\vartheta_{-(m+p_i)U/3,(m+p_i)/3}
 \left(\tau,\int_{P_i}^{p}\rho^*(\Psi)\right)}
{\vartheta_{-(m+p_i)U/3,(m+p_i)/3}
 \left(\tau,\int_{P_i}^{p}\Psi\right)} \\
&=&
\frac{\vartheta_{-(m+p_i)U/3,(m+p_i)/3}
 \left(\tau,\int_{P_i}^{p} \Psi(-U\tau-I_2)^{-1}\right)}
{\vartheta_{-(m+p_i)U/3,(m+p_i)/3}
 \left(\tau,\int_{P_i}^{p}\Psi\right)} \\
&=&
\frac{\vartheta_{-(m+p_i)U/3,(m+p_i)/3}
 \left(\tau^\#,\zeta^\#\right)}
{\vartheta_{-(m+p_i)U/3,(m+p_i)/3}
 \left(\tau,\zeta\right)}. \\
\end{eqnarray*}
Here we set $\tau^\#=\tau$, $\zeta=\int_{P_i}^{p}\Psi$ and 
$\zeta^\#=\zeta(-U\tau-I_2)^{-1}$, and  
use 
$$\rho^*(\Psi)=(\psi_1,\psi_2) W 
(\tau_B)^{-1}=\Psi(U\tau)=\Psi(-U\tau-I_2)^{-1},
$$
where 
$
\displaystyle
W=\begin{pmatrix} \omega &  0 \\
0 & \omega^2
\end{pmatrix}.
$
Applying Corollary in p.85 and Corollary in p.176 of \cite{I} to 
$\sigma=\rho^{-1}=\begin{pmatrix} O & U \\ -U& -I_2\end{pmatrix}\in Sp_2(\mathbf{Z})$, 
we can compute the limit of the last row as $p\to P_i$, and 
$$
\lim_{p\to P_i}\frac{\rho^*(\vartheta_{[m]}^i(p))}
{\vartheta_{[m]}^i(p)}=
\exp\left(\frac{2\pi\sqrt{-1}}{3}(-1)^{r_i}(m+p_i)U\;^t (m+p_i)\right).
$$
The fundamental property yields this lemma.
\qed\bigbreak

Lemma \ref{lem:rho-act} yields a list of orders of zero of $\vartheta_{[m]}(p)$
at $P_i$ $(i=0,1,t,\infty)$ modulo $3$.

\begin{table}[h]
  \centering
  \begin{tabular}[center]{|c|c|c|c|c|}
\hline
$[m]\diagdown p$ & $P_0$ & $P_1$ & $P_t$ & $P_\infty$\\
\hline
$[2,1]$ & $2$ & $0$ & $0$ & $0$ \\
$[0,1]$ & $0$ & $2$ & $0$ & $0$ \\
$[2,0]$ & $0$ & $0$ & $2$ & $0$ \\
$[0,0]$ & $0$ & $0$ & $0$ & $2$ \\
\hline
$[1,1]$ & $1$ & $1$ & $0$ & $0$ \\
$[2,2]$ & $1$ & $0$ & $1$ & $0$ \\
$[0,2]$ & $0$ & $1$ & $0$ & $1$ \\
$[1,0]$ & $0$ & $0$ & $1$ & $1$ \\
\hline
$[1,2]$ & $2$ & $2$ & $2$ & $2$ \\
\hline
  \end{tabular}
\vspace{2mm}
  \caption{List of orders of zero of $\vartheta_{[m]}(p)$}
  \label{tab:zero}
\end{table}

\begin{proof}[Proof of Theorem \ref{th:image}]
We fix an element $\eta \in \bold B$.  By Table \ref{tab:zero}, 
$\vartheta_{[1,2]}(p)=\vartheta_{(-2,-1,1,2)/3}(\tau(\eta),\zeta(p))$ has zeros 
more then 2 (see \cite{M}). Riemann's theorem implies that 
this function of $p\in C$ is identically zero.  
Since the theta divisor is irreducible,  
the image of the period $[\zeta(p)]\in J(C)$ coincides with
$$\{[\zeta]\in J(C)\mid \vartheta_{(-2,-1,1,2)/3}(\tau(\eta),\zeta)=0\}.$$
 \end{proof}
 
We have the following inverse period map for pairs $(C,p)$ of
curve $C$ and a point $p$ on it.
\begin{theorem}
\label{th:u-exp}
Let $p=(u,w)$ be a point of $C$.
The meromorphic function $u$ on $C$ is expressed as
\begin{equation}
\label{eq:u-express}
u=1-\frac{\vartheta_{[0,1]}^3(p)}{\vartheta_{[0,2]}^3(p)}
=1-\frac{\vartheta_{(-1,0,0,1)/3}^3(\tau(\eta),\zeta(p))}
{\vartheta_{(-2,0,0,2)/3}^3(\tau(\eta),\zeta(p))}.
\end{equation}
In particular, 
$$\frac{1}{t}=1-\frac{\vartheta_{(0,0,0,0)/3}^3(\tau(\eta))}
{\vartheta_{(-1,0,0,1)/3}^3(\tau(\eta))}.
$$
\end{theorem}
\begin{remark}
The above formula for $t$ gives the inverse of the Schwartz map 
referred in Remark \ref{rem Schwarz map}.
\end{remark}

\begin{proof}
(1)
The divisor of the meromorphic function $1-u$ is $3P_1-3P_\infty$. 
By Table \ref{tab:zero}, 
$\vartheta_{[0,1]}(p)^3/\vartheta_{[0,2]}(p)^3$ has the same divisor.
Thus their ratio is a constant.
By putting $p=P_0$, we see that this constant is $1$. 
We obtain the expression of $1/t$ by  
substituting $p=P_t$ into (\ref{eq:u-express}) and using the formula
$$\vartheta_{a,b}(\tau,c\tau+d)
=\exp[-\pi\sqrt{-1}(c\tau \;^t c+2c\;^t(b+d))]\vartheta_{a+c,b+d}(\tau,0)$$ 
for $a,b,c,d\in \mathbf{Q}^2$. 
\end{proof}

\subsection{Inverse period map for configuration space}
Using Theorem \ref{th:u-exp}, 
we give the inverse period map 
$V(\vartheta)/G(1-\rho) \to \Cal M$
in terms of theta functions.
\begin{theorem}
\label{th:x1,x2-exp}
Let $((x_1,x_2),\mu)$ be an element in $\Cal M_{mk}$
and $(\eta,\zeta)=per((x_1,x_2),\mu ) \in \Cal D$ 
be the image of $(x_1, x_2, \mu)$ under the period map $per$ defined in 
Definition \ref{modular embedding}.
Let $\tau=\tau(\eta)$ be an element in $\bold H_2$ defined in 
\S \ref{def:tau from eta}.

Then the element $(x_1,x_2)\in \Cal M$ is
obtained by the following theta values:
\begin{align*}
x_1&=\frac
{\vartheta_{(-1,0,0,1)/3}^3(\tau,\zeta)}
{\vartheta_{(-2,0,0,2)/3}^3(\tau,\zeta)}, 
\\
x_2&=\frac
{\vartheta_{(-1,0,0,1)/3}^3(\tau,\iota^*(\zeta))}
{\vartheta_{(-2,0,0,2)/3}^3(\tau,\iota^*(\zeta))},
\end{align*}
where 
$$\iota(\tau,\zeta)=(\tau,\iota^*(\zeta)),\quad 
\iota^*(\zeta)=\zeta\begin{pmatrix}
1 & 0\\
0 &-1
\end{pmatrix}
+\left(0,\frac{\sqrt3\mathbf{i}}{3} \eta\right).
$$
\end{theorem}

\begin{proof}
By the condition of $x_1,x_2$, we have $x_1x_2(1-x_1)(1-x_2)(1-x_1-x_2)\ne0$. 
We have a curve $C_t$ for 
$t=\dfrac{1-x_1-x_2}{(1-x_1)(1-x_2)}\in (0,1)$
and a point $p\in C_t$ whose $u$-coordinate is $1-x_1$.
Apply to (\ref{eq:u-express}) for the point $p\in C_t$.
Then we have 
$$1-x_1=1-\frac{\vartheta_{(-1,0,0,1)/3}^3(\tau,\zeta)}
{\vartheta_{(-2,0,0,2)/3}^3(\tau,\zeta)}.
$$
This gives an expression of $x_1$ in terms of theta values at 
$(\tau,\zeta)$. 
To obtain the expression of $x_2$, we use theta values at 
$(\tau,\iota^*(\zeta))$ instead of $(\tau,\zeta)$ in the above expression.   
The vector $\iota^*(\zeta)$ is computed by 
$\dot y_3'$ and  $\dot y_4'$ in stead of $y_3'$ and  $y_4'$, where 
$$\dot y_3'=\int_{\widetilde \Gamma_4}\Omega_X
 \quad 
 \dot y_4'
 =\int_{\widetilde \Gamma_3}\Omega_X-\omega^2\int_{\widetilde \Gamma_2}\Omega_X.
$$ 
Comparing with (\ref{equation gamma1-4}), 
(\ref{definition of y_i}) and (\ref{scale change}),
we have 
\begin{eqnarray*}
\iota^*(\zeta)
&=&\frac{1}{2}\left(\frac{\dot y_3'+\dot y'_4}{y_1},
\frac{\dot y'_3-\dot y_4'}{y_2}\right)\\
&=&\frac{1}{2}\left(\frac{y_3'+y'_4}{y_1},
\frac{y'_4-y_3'+(2\omega^2/(\omega-1))y_1 }{y_2}\right)\\
&=&\frac{1}{2}\left(\frac{y_3'+y'_4}{y_1},
-\frac{y'_3-y_4'}{y_2}\right)
+\left(0,\frac{\sqrt{3}\mathbf{i}}{3}\eta\right).
\end{eqnarray*}

\end{proof}

\end{document}